%% file: UnErgodicNewEls.tex
\documentclass[]{elsarticle}

\usepackage[T1]{fontenc}

\usepackage{amsmath,amssymb,amsthm}

\usepackage{tikz}
\usetikzlibrary{arrows}
\usepackage{algorithm}
\usepackage[noend]{algpseudocode}
\usepackage{comment}

\theoremstyle{plain}
\newtheorem{lemma}{Lemma}
\newtheorem{theorem}[lemma]{Theorem}
\newtheorem{proposition}[lemma]{Proposition}
\newtheorem{corollary}[lemma]{Corollary}

\theoremstyle{definition}
\newtheorem{definition}[lemma]{Definition}

\theoremstyle{remark}
\newtheorem{remark}[lemma]{Remark}
\newtheorem{claim}[lemma]{Claim}
\newtheorem{observation}[lemma]{Observation}

\newcommand{\mtt}[1]{\mathtt{#1}}
\newcommand{\mbf}[1]{\mathbf{#1}}
\newcommand{\R}{\mathbb{R}}
\newcommand{\Q}{\mathbb{Q}}
\newcommand{\Z}{\mathbb{Z}}
\newcommand{\C}{\mathbb{C}}
\newcommand{\N}{\mathbb{N}}

\newcommand{\B}{\mathcal{B}}
\newcommand{\M}{\mathcal{M}}
\newcommand{\Oh}{\mathcal{O}}
\newcommand{\Pow}{\mathcal{P}}
\newcommand{\INF}{{}^\infty}

\newcommand{\alp}{\Sigma}
\newcommand{\varalp}{\Gamma}
\newcommand{\vvaralp}{\Delta}
\newcommand{\CA}{\Psi}
\newcommand{\varCA}{\Phi}
\newcommand{\vvarCA}{\Theta}
\newcommand{\locf}{\psi}

\newcommand{\intp}{\mathrm{Int}}
\newcommand{\univ}{\mathrm{Univ}}
\newcommand{\tra}{\mathcal{T}}
\newcommand{\size}[1]{\left\| #1 \right\|}

\newcommand{\addr}{\mathtt{Addr}}
\newcommand{\age}{\mathtt{Age}}
\newcommand{\simu}{\mathtt{Sim}}
\newcommand{\live}{\mathtt{Live}}
\newcommand{\work}{\mathtt{Work}}
\newcommand{\rmail}{\mathtt{RMail}}
\newcommand{\lmail}{\mathtt{LMail}}
\newcommand{\prog}{\mathtt{Prog}}
\newcommand{\out}{\mathtt{Out}}
\newcommand{\lev}{\mathtt{Level}}
\newcommand{\doom}{\mathtt{Doomed}}
\newcommand{\rbor}{\mathtt{RBord}}
\newcommand{\lbor}{\mathtt{LBord}}
\newcommand{\kind}{\mathtt{Kind}}
\newcommand{\cou}{\mathtt{Cnt}}
\newcommand{\ccou}{\mathtt{Cnt2}}
\newcommand{\celem}{\mathtt{Elem}}

\title{A Uniquely Ergodic Cellular Automaton\tnoteref{thanks}}
\tnotetext[thanks]{Research supported by the Academy of Finland Grant 131558}
\author[utu]{Ilkka T\"orm\"a}
\ead{iatorm@utu.fi}
\address[utu]{
TUCS -- Turku Centre for Computer Science \\
Department of Mathematics and Statistics \\
20014 University of Turku, Finland \\
+358 2 333 6012
}

\begin{document}

\begin{abstract}
We construct a one-dimensional uniquely ergodic cellular automaton which is not nilpotent.
This automaton can perform asymptotically infinitely sparse computation, which nevertheless never disappears completely.
The construction builds on the self-simulating automaton of G\'acs.
We also prove related results of dynamical and computational nature, including the undecidability of unique ergodicity, and the undecidability of nilpotency in uniquely ergodic cellular automata.
\end{abstract}


\maketitle

\section{Introduction}
\label{sec:Intro}

Cellular automata form a class of discrete dynamical systems that are simple to define but often surprisingly hard to analyze.
They consist of an infinite array of cells, arranged in a regular lattice of one or more dimensions, each of which contains a state drawn from a finite set.
The dynamics is given by a local function that synchronously assigns each cell a new state based on the previous states of itself and its neighbors.
Cellular automata can be used as idealized models of many physical and biological processes, or as massively parallel computing devices, and their history goes back to the 60's \cite{He69}.

Since cellular automata are so discrete and finitely definable, there are many natural decidability problems concerning their dynamical properties.
The most classical result in this area is the undecidability of nilpotency \cite{Ka92}.
A cellular automaton is nilpotent if there is a number $n \in \N$ such that, regardless of the initial contents of the cells, every cell is in the same `blank' state after $n$ timesteps.
There are several variants of this property, some of which are equivalent to nilpotency, and some of which are not.
If $n$ can depend on the initial configuration and on an individual cell, then the automaton is nilpotent (proved in \cite{GuRi08} for one-dimensional automata, and generalized in \cite{Sa12c} to all dimensions).
 In \cite{BoPoTh06}, a notion called $\mu$-quasi-nilpotency was defined for a probability measure $\mu$, using the concept of $\mu$-limit sets defined in \cite{KuMa00}.
Simple examples show that for many natural measures $\mu$, there exist $\mu$-quasi-nilpotent cellular automata which are not nilpotent.

In this article, we consider the notion of \emph{unique ergodicity}, which originates from ergodic theory.
A dynamical system is uniquely ergodic if there is a unique probability measure on the state space which is invariant under the dynamics.
For cellular automata with a blank state, this is equivalent to the condition that every cell of every initial configuration asymptotically spends only a negligible number of timesteps in a non-blank state.
Our main result is the construction of a uniquely ergodic cellular automaton which is not nilpotent.
We also define the notion of \emph{asymptotic nilpotency in density}, which means that the density of non-blank cells of any initial configuration converges to $0$, as the cellular automaton is applied repeatedly.
Our uniquely ergodic automaton also has this property, while the two notions are independent in general.

Our construction is based on the error-resistant cellular automaton of G\'acs \cite{Ga01} (the article is extremely long and complicated, see \cite{Gr01} for a more reader-friendly outline).
In particular, we use the powerful and adaptable tool of \emph{programmatic self-simulation}.
Self-similarity itself is not a new method in the field of symbolic dynamics and cellular automata, since the aperiodic tileset of Robinson, together with his proof of undecidability of the tiling problem, is based on geometric self-similarity \cite{Ro71}.
Another classical example is the tileset of Mozes, which realizes a two-dimensional substitution system \cite{Mo89}.
Programmatic self-simulation differs from these in that the self-similarity does not arise from geometric constructs, but from the system explicitly computing its own local structure and simulating it on a larger scale.
In recent years, programmatic self-simulation has been successfully applied to numerous problems, for example to realize every effective closed one-dimensional subshift by a tiling system, and to construct a tiling system whose every valid tiling has high Kolmogorov complexity (see \cite{DuRoSh12} and references therein).
The possible topological entropies of cellular automata in all dimensions have also been characterized using the method \cite{GuZi13}.

The main idea of this paper is the following.
We construct a cellular automaton which realizes approximate programmatic self-simulation.
This means that in the time evolution of certain well-formed configurations, the cells are organized into `macro-cells' that collectively behave similarly to the original system, so the macro-cells are again organized into `macro-macro-cells' and so on.
We construct the automaton so that it is \emph{sparse} in the following sense: any given cell will spend at least a certain percentage of its time in the blank state.
Each simulation level is also sparser than the previous, so that the macro-cells of a very high simulation level are blank almost all the time.
Those macro-cells that represent blank cells consist of blank cells themselves, so that in the limit, every cell is almost always blank.
We also guarantee that all patterns except the correct simulation structure are destroyed and replaced by blank cells.
The construction is split into two relatively independent parts: a computable transformation that turns any cellular automaton into a sparse version of itself, and a general self-simulating automaton on which we apply the transformation.
The main reason for the complexity of the construction is that programmatic self-simulation is inherently very complicated and its implementation contains lots of technical details, even more so when extra conditions are required from the system.
The length of our proof and the amount of definitions, on the other hand, are mainly due to the last condition that all incorrect structure must eventually disappear.

This paper is organized as follows.
Section~\ref{sec:Defs} consists of the (mostly) standard definitions and notation needed in this article.
Section~\ref{sec:Prelim} contains the characterization of unique ergodicity for cellular automata with a quiescent state, and is the only section containing general ergodic theory.
A reader not interested in ergodic theory may skip the proofs of this section, and use our characterization as the definition of unique ergodicity for cellular automata.
In Section~\ref{sec:Simu}, we define several notions related to simulations between cellular automata, and prove their basic properties.
Since our construction uses programmatic self-simulation, and involves quite complicated automata, Section~\ref{sec:Interpreter} presents a simple programming language for expressing CA rules.
We also define CA transformations, which are functions operating on these programs.
Section~\ref{sec:Sparse} describes a certain CA transformation that performs sparsification, that is, modifies its input automaton so that all cells spend only a small (but positive) fraction of their time in non-blank states.

Section~\ref{sec:UnivSim} describes a family of cellular automata that perform universal programmatic simulation.
This is not a formal notion, but refers to the fact that the simulation is carried out using a Turing machine that interprets the program of the target automaton.
In Section~\ref{sec:Amplification}, these universal simulators are used to realize amplification, which is a form of approximate self-simulation parameterized by a CA transformation.
Essentially, amplification produces a cellular automaton that programmatically simulates a transformed version of itself.
In Section~\ref{sec:SparseAmp}, we apply amplification to the sparsification transformation, obtaining an `infinitely sparse' cellular automaton, and prove its unique ergodicity, the main result of this article.
Section~\ref{sec:Further} contains some further definitions and results that can be proved using our amplifier construction, including asymptotic nilpotency in density, the undecidability of unique ergodicity, and the undecidability of nilpotency in the class of uniquely ergodic cellular automata.
Finally, Section~\ref{sec:Conclu} contains our conclusions and some future research directions.

This work is based on the author's Master's thesis \cite{To12}.

\section{Definitions}
\label{sec:Defs}

Let $\alp$ be a finite set, called the \emph{state set}.
For a set $D \subset \Z$ (or $D \subset \Z^2$), a function $s : D \to \alp$ is called a \emph{pattern over $\alp$ with domain $D$}.
A pattern whose domain is an interval (product of intervals) is called a \emph{word} (a \emph{rectangle}, respectively), and a \emph{cell} is a pattern whose domain is a singleton.
Translates of a pattern are sometimes deemed equal, and sometimes not; the choice should always be clear from the context.
In particular, by abuse of language and notation, cells are sometimes treated as elements of $\alp$.
A pattern with domain $\Z$ ($\Z^2$) is called a \emph{one-dimensional configuration} (\emph{two-dimensional configuration}, respectively).
The sets of one- and two-dimensional configurations are denoted $\alp^\Z$ and $\alp^{\Z^2}$, respectively.
If $x \in \alp^\Z$ and $i \in \Z$, we denote by $x_i$ the $i$'th cell of $x$, and if $D \subset \Z$, we denote by $x_D$ the restriction of $x$ to $D$.
For a two-dimensional configuration $\eta \in \alp^{\Z^2}$ and $(i,t) \in \Z^2$, we denote by $\eta^t_i$ the cell of $\eta$ at $(i,t)$, and if $D, E \subset \Z$, we denote by $\eta^E_D$ the restriction of $\eta$ to $D \times E \subset \Z^2$.
Also, for $i, t \in \Z$, we denote by $\eta_i = \eta^\Z_i$ the $i$'th column of $\eta$, and by $\eta^t = \eta^t_\Z$ the $t$'th row of $\eta$.
For a rectangular pattern $P \in \alp^{k m \times \ell n}$ and a function $\pi : \alp^{k \times \ell} \to \varalp$, define $\pi(P) \in \varalp^{m \times n}$ by $\pi(P)^t_i = \pi(P^{[t k, (t+1)k-1]}_{[i \ell, (i+1)\ell-1]})$, and extend this to two-dimensional configurations.
Define the \emph{Cantor metric} on $\alp^\Z$ by
\[ d_C(x,y) = \inf \{ 2^{-n} \;|\; n \in \N, x_{[-n,n]} = y_{[-n,n]} \}. \]

A \emph{subshift} is a topologically closed set $X \subset \alp^\Z$ satisfying $\sigma(X) = X$, where $\sigma : \alp^\Z \to \alp^\Z$ is the \emph{shift map}, defined by $\sigma(x)_i = x_{i+1}$ for all $x \in \alp^\Z$ and $i \in \Z$.
The \emph{language} of $X$, denoted by $\B(X)$, is the set of words in $\alp^*$ that occur as a subpattern in some configuration of $X$.
For $\ell \in \N$, we denote $\B_\ell(X) = \B(X) \cap \alp^\ell$.

A \emph{cellular automaton} is a mapping $\CA : X \to X$ on a subshift $X \subset \alp^\Z$ defined by a \emph{local function} $\locf : \B_{2r+1}(X) \to \alp$ such that $\CA(x)_i = \locf(x_{i-r}, \ldots, x_{i+r})$ for all $x \in X$ and $i \in \Z$.
Alternatively, cellular automata are exactly the $d_C$-continuous functions from $X$ to itself that commute with the shift map $\sigma$ \cite{He69}.
Note that the shift map is a cellular automaton itself.
Abusing notation, a cellular automaton can also be applied to a word $w \in \alp^n$ if $n \geq 2r$, the result being $\CA(w) = \locf(w_{[0, 2r]}) \locf(w_{[1, 2r+1]}) \cdots \locf(w_{[n-2r-1, n-1]}) \in \alp^{n-2r}$.
A state $B \in \alp$ is \emph{quiescent} for $\CA$ if $B^{2r+1} \in \B(X)$ and $\locf(B^{2r+1}) = B$.
Unless otherwise noted, we only consider automata with $X = \alp^\Z$, $r = 1$ and state set $\alp = \{0,1\}^K$ for some $K \in \N$.
In our constructions, we will partition the integer interval $[0,K-1]$ into disjoint subintervals $[k_0,k_1-1], [k_1,k_2-1], \ldots, [k_{n-1},k_n-1]$, where $k_0 = 0$ and $k_n = K$.
Then each state vector $v \in \{0,1\}^K$ can be expressed as $v = w_0 w_1 \cdots w_{n-1}$, where $w_i \in \{0,1\}^{k_{i+1}-k_i}$.
The subintervals $[k_i,k_{i+1}-1]$ are called \emph{fields}, and the binary vectors $w_i$ are the \emph{values} of those fields for the state $v$.
If a field is given a name, say $[k_i,k_{i+1}-1] = \mathtt{Field}_i$, then we denote $w_i = \mathtt{Field}_i(v)$.
We will also denote $\size{\CA} = \size{\alp} = K$, and $B_\alp = 0^{\size{\alp}}$ (the \emph{blank state} of $\alp$).
\emph{Unless otherwise noted, we always assume the blank state to be quiescent.}

Let $\CA : \alp^\Z \to \alp^\Z$ be a CA with local function $\locf$.
We define two \emph{nondeterministic cellular automata} associated to $\CA$, denoting the powerset of a set $X$ by $\Pow(X)$.
First, $\tilde \CA : \alp^\Z \to \Pow(\alp^\Z)$ is defined by
\[ \tilde \CA(x) = \{ y \in \alp^\Z \;|\; \forall i \in \Z : y_i = \locf(x_{i-1}, x_i, x_{i+1}) \vee y_i = B_\alp \}. \]
Second, for all triples $(a,b,c) \in \alp^3$, define $f(a,b,c) \subset \alp^3$ as the set of triples obtained by setting a subset of the coordinates of $(a,b,c)$ to the blank state $B_\alp$.
Then define $\hat \CA : \alp^\Z \to \Pow(\alp^\Z)$ by
\[ \hat \CA(x) = \{ y \in \alp^\Z \;|\; \forall i \in \Z : \exists v \in f(x_{i-1}, x_i, x_{i+1}) : y_i = \locf(v) \}. \]
Intuitively, $\tilde \CA$ behaves like the corresponding deterministic automaton, but in any coordinate it may choose to produce the blank state $B_\alp$ instead of the `correct' value.
The other nondeterministic automaton, $\hat \CA$, has even more freedom, as it can choose to regard any cell in a local neighborhood as blank.
These objects are not standard in the literature, but we will use them in the following way.
The simulations of a CA $\CA$ by another CA $\varCA$ that we construct are not perfect, in the sense that not all initial configurations of $\varCA$ result in a correct simulation of $\CA$.
However, in our case they always result in the simulation of $\tilde \CA$ or $\hat \CA$, which are asymptotically at least as sparse as $\CA$.

Let $\CA$ be a cellular automaton on $\alp^\Z$.
We say that a two-dimensional configuration $\eta \in \alp^{\Z^2}$ (or $\alp^{\Z \times \N}$) is a (one-directional) \emph{trajectory of $\CA$} if $\eta^{t+1} = \CA(\eta^t)$ for all $t \in \Z$ ($t \in \N$, respectively).
The set of all (one-directional) trajectories of $\CA$ is denoted by $\tra_\CA$ ($\tra^+_\CA$).
Trajectories of the associated nondeterministic automata are defined analogously, and the corresponding sets are denoted by $\tilde \tra_\CA$, $\tilde \tra^+_\CA$, $\hat \tra_\CA$ and $\hat \tra^+_\CA$.
Clearly $\tra_\CA \subset \tilde \tra_\CA \subset \hat \tra_\CA$ holds, as does the analogous inclusion chain for one-directional trajectories.
Note that each two-directional trajectory can also be seen as a one-directional trajectory by ignoring the cells with negative $t$-coordinate.
The set $\Omega_\CA = \bigcap_{n \in \Z} \CA^n(\alp^\Z)$ is called the \emph{limit set} of $\CA$.
An alternative characterization for the limit set is $\Omega_\CA = \{ \eta^0 \;|\; \eta \in \tra_\CA \}$ \cite{CuJaYu89}.

A cellular automaton $\CA$ is called \emph{nilpotent} if there exists $n \in \N$ such that $\CA^n(x) = \INF B_\alp^\infty$ for all $x \in \alp^\Z$.
This is equivalent to the condition $\Omega_\CA = \{\INF B_\alp^\infty\}$.

A \emph{topological dynamical system} is a tuple $(X,T)$, where $X$ is a compact metric space and $T : X \to X$ a continuous map.
A probability measure $\mu$ on the Borel subsets of $X$ is \emph{$T$-invariant}, if $\mu(Y) = \mu(T^{-1}(Y))$ holds for all Borel sets $Y \subset X$.
The set of $T$-invariant probability measures on $X$ is denoted by $\M_T$.
It is known that this set is always nonempty \cite{Wa82}.
The map $T$, or the system $(X,T)$, is called \emph{uniquely ergodic} if $|\M_T| = 1$.

We define the \emph{arithmetical hierarchy} of logical formulae over $\N$ as follows.
A \emph{bounded quantifier} has the form $\forall n < \tau$ or $\exists n < \tau$, where $\tau$ is a term which does not contain $n$ (but may contain other free variables).
A first-order arithmetical formula whose every quantifier is bounded has the classifications $\Pi^0_0$ and $\Sigma^0_0$.
If $\phi$ is $\Sigma^0_n$ ($\Pi^0_n$), then all formulae logically equivalent to $\forall n: \phi$ ($\exists n: \phi$) are $\Pi^0_{n+1}$ ($\Sigma^0_{n+1}$, respectively).
A set $X \subset \N$ gets the same classification as a formula $\phi$, if $X = \{ n \in \N \;|\; \phi(n) \}$ holds.
It is known that the class of $\Sigma^0_1$-subsets of $\N$ consists of exactly the recursively enumerable sets, and $\Pi^0_1$ of their complements.
When classifying sets of objects other than natural numbers, such as cellular automata, we assume that the objects are in a natural and computable bijection with $\N$.
For a general introduction to the topic, see \cite[Chapter IV.1]{Od89}.

For two sets $X, Y \subset \N$, we say that $Y$ is \emph{reducible to $X$} (also known as \emph{many-one reducible}) if there exists a computable function $f : \N \to \N$ such that $f(Y) \subset X$ and $f(\N \setminus Y) \subset \N \setminus X$.
For a class $\Xi$ of subsets, we say that $X$ is \emph{$\Xi$-hard} if every $Y$ in $\Xi$ is reducible to $X$.
If $X$ also lies in $\Xi$, we say that $X$ is \emph{$\Xi$-complete}.
Note that the famous class of \emph{NP-complete problems} does not fit into this definition, since the definition of \emph{polynomial reducibility} further assumes that the function $f$ is computable in polynomial time.

The operator ${\bmod}$ has two related but completely orthogonal uses in this paper.
First, if $i, j \in \Z$ and $n \in \N$, the notation $i \equiv j \bmod n$ is equivalent to the formula $\exists k \in \Z: i - j = kn$.
Here, ${\bmod}$ is used in conjunction with ${\equiv}$ to form a ternary predicate.
Second, if $i \in \Z$ and $n \in \N$, then $i \bmod n$ denotes the unique number $j \in [0,n-1]$ with $i \equiv j \bmod n$.
In this case, ${\bmod}$ is a binary function from $\Z \times \N$ to $\N$.

This paper contains several pictures of spacetime diagrams of cellular automata.
As some of the automata contain Turing machine computations, and it is customary to illustrate these using a time axis that increases upwards, we have chosen this representation also for our spacetime diagrams.

\section{Preliminary Results}
\label{sec:Prelim}

In this section, we state some preliminary results about nilpotency and unique ergodicity in cellular automata.
We begin with a characterization of nilpotency.
The first and third claims are proved in \cite{CuJaYu89}, and the second in \cite{GuRi08}.
We do not use this result explicitly, and it is meant to be a formal version of the remarks on nilpotency in Section~\ref{sec:Intro}.

\begin{proposition}
\label{prop:Nilpotent}
Let $\CA : \alp^\Z \to \alp^\Z$ be a cellular automaton with quiescent state $B_\alp$. The following conditions are equivalent to the nilpotency of $\CA$.
\begin{itemize}
\item For all $x \in \alp^\Z$, there exists $n \in \N$ such that $\CA^n(x) = \INF B_\alp^\infty$.
\item For all $x \in \alp^\Z$ and $i \in \Z$, there exists $n \in \N$ such that $\CA^m(x)_i = B_\alp$ for all $m \geq n$.
\item The limit set of $\CA$ is a singleton: $\Omega_\CA = \{\INF B_\alp^\infty\}$.
\end{itemize}
\end{proposition}

Next, let us consider unique ergodicity.
Of course, a nilpotent cellular automaton $\CA : \alp^\Z \to \alp^\Z$ is uniquely ergodic, since $\M_\CA$ consists of the Dirac measure $\mu_0$ defined by
\begin{equation}
\label{eq:Dirac}
\mu_0(Y) = \left\{ \begin{array}{cc}
	1, & \mbox{if~} \INF B_\alp^\infty \in Y \\
	0, & \mbox{if~} \INF B_\alp^\infty \notin Y.
\end{array} \right.
\end{equation}
General invariant measures can be tricky to work with, so we prove a combinatorial characterization (Corollary~\ref{cor:UEChara}) of unique ergodicity for cellular automata with a quiescent state.
Uniquely ergodic cellular automata turn out to have a strong nilpotency-like property.
First, we show that it suffices to consider the limit set.

\begin{lemma}
\label{lem:LimitSet}
A cellular automaton $\CA : \alp^\Z \to \alp^\Z$ is uniquely ergodic if and only if its restriction $\CA|_{\Omega_\CA}$ to the limit set is.
\end{lemma}

\begin{proof}
Denote $\varCA = \CA|_{\Omega_\CA}$.
If $\CA$ is uniquely ergodic, then clearly $\varCA$ must also be, since any invariant measure of the latter is also in $\M_\CA$.

Let then $\mu \in \M_\CA$ be arbitrary.
We will prove that $\mu(\Omega_\CA) = 1$, showing that the restriction of $\mu$ to subsets of $\Omega_\CA$ is in $\M_\varCA$.
For that, let
\[ X_n = \{ x \in \alp^\Z \;|\; \CA^{-n}(x) \neq \emptyset, \CA^{-n-1}(x) = \emptyset \} \subset \alp^\Z \]
for all $n \in \N$.
Clearly, the sets $X_n$ are pairwise disjoint, and cover the set $\alp^\Z - \Omega_\CA$.
We also have that $\mu(X_n) = \mu(\CA^{-n-1}(X_n)) = 0$ for all $n$, and it follows that $\mu(\Omega_\CA) = \mu(\alp^\Z) - \mu(\bigcup_{n \in \N} X_n) = 1$.
Thus if $\CA$ is not uniquely ergodic, then neither is $\varCA$.
\end{proof}

We need the following classical result of ergodic theory:

\begin{lemma}[Theorem 6.19 of \cite{Wa82}]
\label{lem:UEChara}
Let $(X,T)$ be a topological dynamical system.
The following conditions are equivalent:
\begin{enumerate}
\item  The map $T$ is uniquely ergodic.
\item There exists a measure $\mu \in \M_T$ such that $(\frac{1}{n} \sum_{i=0}^{n-1} f(T^i(x)))_{n \in \N}$ converges to $\int_X \! f \, d \mu$ for every $x \in X$ and continuous $f : X \to \C$.
\item For every continuous $f : X \to \C$, the sequence $(\frac{1}{n} \sum_{i=0}^{n-1} f \circ T^i)_{n \in \N}$ converges uniformly to a constant function.
\end{enumerate}
\end{lemma}

For the next result and its proof, we make the following definition.

\begin{definition}
Let $X \subset \alp^\Z$ be a subshift, let $\CA : X \to X$ be a cellular automaton, let $w \in \B_{2r+1}(X)$, and let $x \in X$.
We denote
\[ d_\CA(w, x) = \liminf_{n \rightarrow \infty} \frac{1}{n} |\{ k \in [0,n-1] \;|\; \CA^t(x)_{[-r,r]} = w \}|, \]
the lower asymptotic density of the occurrences of the word $w$ at the origin.
\end{definition}

The combinatorial characterization of uniquely ergodic cellular automata is the following.

\begin{proposition}
\label{prop:UEChara}
Let $X \subset \alp^\Z$ be a subshift with $\INF B_\alp^\infty \in X$, and let $\CA : X \to X$ be a cellular automaton.
Then $\CA$ is uniquely ergodic if and only if $d_\Psi(B_\alp, x) = 1$ holds for all $x \in X$.
If this is the case, the convergence of the limit is uniform in $x$.
\end{proposition}

\begin{proof}
Suppose first that $d_\Psi(B_\alp, x) = 1$ holds for all $x \in X$.
Let $\mu_0$ be the Dirac probability measure on $X$ defined by \eqref{eq:Dirac}.
Then $\mu_0$ is $\CA$-invariant, since $B_\alp$ is quiescent, and $\int_X \! f \, d \mu_0 = f(\INF B_\alp \INF)$ holds for all continuous functions $f : X \to \C$.

Let $x \in X$ and $f : X \to \C$ continuous.
For each $n \in \N$, let $f_n : X \to \C$ be a (continuous) function of the form $\sum_{i=1}^k a_i \boldsymbol{1}_{C_i}$ with $a_i \in \C$ and $C_i \subset X$ a cylinder set (a set of the form $\{ y \in X \;|\; y_{[-r,r]} = w \}$ for a word $w \in \B_{2r+1}(X)$), and such that $|f(y)-f_n(y)| < \frac{1}{n}$ for all $y \in X$.
This is possible, since functions of this form are dense in the uniform topology.

The condition $d_\CA(B_\alp, x) = 1$ implies that $d_\CA(B_\alp^{2r+1}, x) = 1$ for all $r \in \N$, and thus for all $n \in \N$, the average $\frac{1}{m} \sum_{i=0}^{m-1} f_n(\CA^i(x))$ converges to the constant $f_n(\INF B_\alp \INF)$ as $m$ grows.
Hence also $\frac{1}{m} \sum_{i=0}^{m-1} f(\CA^i(x))$ converges to $f(\INF B_\alp \INF)$, which can be proved using a standard $\frac{\epsilon}{3}$-approximation.
Unique ergodicity now follows by Lemma~\ref{lem:UEChara}.

Let then $\CA$ be uniquely ergodic, and consider the (again continuous) function $f : \alp^\Z \to \{0,1\}$ defined by $f^{-1}(1) = \{ y \in \alp^\Z \;|\; y_0 = B_\alp \}$.
By Lemma~\ref{lem:UEChara}, the sequence $(\frac{1}{n} \sum_{i=0}^{n-1} f \circ \CA^i)_{n \in \N}$ converges uniformly to a constant, which must be $1$ because of the fixed point $\INF B_\alp \INF$.
This implies that $d_\CA(B_\alp,x) = 1$, and that the convergence is uniform in $x$.
\end{proof}

Lemma~\ref{lem:LimitSet} and the above together imply the following.

\begin{corollary}
\label{cor:UEChara}
A cellular automaton $\CA : \alp^\Z \to \alp^\Z$ is uniquely ergodic if and only if $d_\Psi(B_\alp, x) = 1$ holds for all $x \in \Omega_\CA$, and then the convergence of the limit is uniform in $x$.
\end{corollary}

Another way to state Corollary~\ref{cor:UEChara} is that a CA $\CA : \alp^\Z \to \alp^\Z$ is uniquely ergodic if and only if the asymptotic density of $B_\alp$-cells in the central column $\eta^{[0,\infty)}_0$ is $1$ for every $\eta \in \tra_\CA$.
Since this then holds for all columns, the automaton $\CA$ acts in a very `sparse' environment, forming large areas of blank cells almost everywhere.

\section{Simulation}
\label{sec:Simu}

In this section, we define the different notions of simulation between cellular automata that we use in the course of this paper.
We begin with the definition of basic simulation, which is somewhat nonstandard, since it is specifically tailored to our needs.

\begin{definition}
Let $\CA : \alp^\Z \to \alp^\Z$ and $\varCA : \varalp^\Z \to \varalp^\Z$ be two cellular automata.
A \emph{simulation} of $\varCA$ by $\CA$ is a quintuple $(Q, U, (a,b), C, \pi)$, where $Q, U \in \N$ are the \emph{dimensions} of the simulation, $(a,b) \in [0,Q-1] \times [0,U-1]$ is the \emph{base coordinate}, $C \subset \alp^{Q \times U}$ is the set of \emph{macro-cells} and $\pi : \Sigma^{Q \times U} \to \varalp$ is the \emph{state function}, such that $\pi(\alp^{Q \times U} \setminus C) = B_\varalp$ and $\tra^+_\varCA \subset \pi(\tra^+_\CA)$ hold, as does $P^b_a \neq B_\alp$ for all $P \in C$.
If such a simulation exists, we say that $\CA$ \emph{$(Q, U)$-simulates} $\varCA$.
The base coordinate will be dropped from the definition if it is unnecessary or clear from the context.

An \emph{amplifier} is a sequence $(\CA_n)_{n \in \N}$ of cellular automata such that for each $n \in \N$, the automaton $\CA_n$ simulates $\CA_{n+1}$.
The \emph{bottom} of the amplifier is the automaton $\CA_0$.
\end{definition}

A simulation is a method of constructing, for each trajectory of $\varCA$, a corresponding trajectory of $\CA$, where the correspondence is given by a function that associates a state of $\varCA$ to every $Q \times U$-rectangle of $\CA$-states.
An amplifier, especially one where the $\CA_n$ are somewhat similar, is a form of \emph{approximate self-simulation}.
The term amplifier is also used in \cite{Ga01} with a similar meaning, while the notion of simulation in said paper is much more general.
Note that the base coordinate $(a,b)$ and the set $C$ of macro-cells are largely irrelevant in the definition.
Next, we define some more refined classes of simulations that depend more heavily on $(a,b)$ and $C$.

\begin{definition}
Let $(Q, U, C, \pi)$ be a simulation of $\varCA : \varalp^\Z \to \varalp^\Z$ by $\CA : \alp^\Z \to \alp^\Z$.
We say that the simulation is \emph{rigid}, if $\pi(\tilde \tra_\CA) \subset \tilde \tra_\varCA$, \emph{weakly rigid}, if $\pi(\tilde \tra_\CA) \subset \hat \tra_\varCA$, and \emph{strongly rigid}, if $\pi(\hat \tra_\CA) \subset \tilde \tra_\varCA$.
\end{definition}

Intuitively, a simulation is rigid if non-blank macro-cells cannot appear from nothing, but must be the result of locally correct simulation by other macro-cells, in nondeterministic two-directional trajectories.
Weak and strong rigidity are technical variations of the concept.
We will only construct simulations that are at least weakly rigid, since they are much easier to reason about than general simulations, where the $\pi$-image of a trajectory of $\CA$ may not be related to $\varCA$ in any way.
In our case, even the trajectories of the nondeterministic CA $\tilde \varCA$ and $\hat \varCA$ will be sparse, and a rigid simulation transfers this sparsity to $\CA$.
In the course of the proof, we construct an amplifier $(\CA_n)_{n \in \N}$ with rigid simulations, where each $\tilde \CA_n$ is sparser than the previous one.
Since this sparsity is transferred to the bottom automaton $\CA_0$ as described above, it is uniquely ergodic.

\begin{definition}
Let $\CA : \alp^\Z \to \alp^Z$ be a CA with local function $\locf$, and let $\delta \in \{-1,0,1\}$.
A state $c \in \alp$ is called \emph{$\delta$-demanding}, if $\locf(c_{-1},c_0,c_1) = c$ implies $c_\delta \neq B_\alp$ for all $c_{-1}, c_0, c_1 \in \alp$.
We assign two directed graphs $G$ and $G'$ to every configuration $\eta \in \alp^{\Z^2}$ as follows.
The vertex set of both of them consists of the non-blank cells of $\eta$.
There is an edge in $G'$ from the cell $\eta^{t-1}_{i+\delta}$ to $\eta^t_i$, where $i, t \in \Z$ and $\delta \in \{-1,0,1\}$, if the cells are not blank, and in $G$ if $\eta^t_i$ is also $\delta$-demanding.

Cells connected by an edge in $G$ ($G'$) are called \emph{(weakly) adjacent}.
Two cells are \emph{(directly) connected} if there is a (directed) path in $G$ between them.
The analogous weak concepts are defined for $G'$.

Let $(Q, U, (a,b), C, \pi)$ be a simulation of $\varCA : \varalp^\Z \to \varalp^\Z$ by $\CA$.
We say that the simulation is \emph{connecting}, if for all trajectories $\eta \in \tilde \tra_\CA$ and all $i, t \in \Z$ and $\delta \in \{-1,0,1\}$ such that $\pi(\eta)^t_i$ and $\pi(\eta)^{t-1}_{i+\delta}$ are adjacent, the base cells $\eta^{tU+b}_{iQ+a}$ and $\eta^{sU+b}_{jQ+a}$ are directly connected.
The simulation is \emph{strongly connecting}, if the above holds for all $ \eta \in \hat \tra_\CA$.
\end{definition}

Intuitively, a connecting simulation respects the existence of directed paths.
See Figure~\ref{fig:Connection} for a visualization of a connecting simulation.
The concepts of adjacent and connected cells are somewhat technical, since two neighboring non-blank cells may not be even indirectly connected, but they simplify our constructions and proofs.
For this reason, they are used more frequently than their more natural weak counterparts.
In particular, connecting simulations with lots of demanding cells enable certain geometric arguments, where we show that the bases of two macro-cells in a trajectory must both be directly connected to some third cell, which implies that the macro-cells are correctly aligned or share some other desired property.

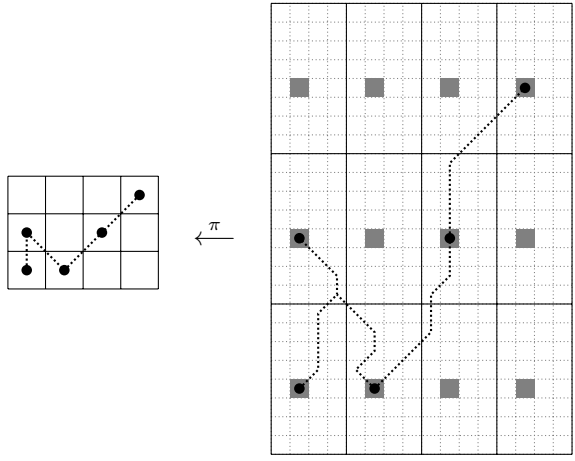
\begin{figure}
\begin{center}
\begin{tikzpicture}

\draw[step=.5] (0,0) grid (2,1.5);

\draw[thick,densely dotted] (.25,.25) -- (.25,.75) -- (.75,.25) -- (1.75,1.25);
\fill (.25,.25) circle (0.07);
\fill (.25,.75) circle (0.07);
\fill (.75,.25) circle (0.07);
\fill (1.25,.75) circle (0.07);
\fill (1.75,1.25) circle (0.07);

\node () at (2.75,.75) {$\stackrel{\pi}{\longleftarrow}$};

\begin{scope}[yscale=2,shift={(3.5,-1.1)}]

\begin{scope}[yscale=.5]
\foreach \x in {0,1,2,3}{
	\foreach \y in {0,1,2}{
		\fill[gray] (\x+.25,2*\y+.75) rectangle ++(.25,.25);
	}
}
\draw[gray,densely dotted,step=.25] (0,0) grid (4,6);
\foreach \x/\y in {0/0,0/1,1/0,2/1,3/2}{
	\fill (\x+.25+.125,2*\y+.75+.125) circle (.07);
}
\begin{scope}[shift={(.125,.125)},xscale=.25,yscale=.25]
	\draw[thick,densely dotted] (1,3) -- ++(1,1) -- ++(0,3) -- ++(1,1) -- ++(0,1) -- ++(-2,2);
	\draw[thick,densely dotted] (5,3) -- ++(-1,1) -- ++(1,1) -- ++(0,1) -- ++(-2,2);
	\draw[thick,densely dotted] (5,3) -- ++(3,3) -- ++(0,2) -- ++(1,1) -- ++(0,2);
	\draw[thick,densely dotted] (9,11) -- ++(0,4) -- ++(4,4);
\end{scope}

\end{scope}

\draw (0,0) grid (4,3);

\end{scope}

\end{tikzpicture}
\end{center}
\caption{A connecting simulation.
The connecting directed paths on the right corresponds to the adjacencies in the $\pi$-image on the right.
The base cells of the macro-cells are shaded.}
\label{fig:Connection}
\end{figure}

We now show that the simulation relation is transitive.

\begin{lemma}
\label{lem:Composition}
Let $\CA : \alp^\Z \to \alp^\Z$, $\varCA : \varalp^\Z \to \varalp^\Z$ and $\vvarCA : \vvaralp^\Z \to \vvaralp^\Z$ be three cellular automata, and suppose that $\CA$ $(Q,U)$-simulates $\varCA$, and $\varCA$ $(Q',U')$-simulates $\vvarCA$.
Then, $\CA$ $(QQ', UU')$-simulates $\vvarCA$.
\end{lemma}

\begin{proof}
Denote by $S = (Q,U,(a,b),C,\pi)$ the simulation of $\varCA$ by $\CA$, and by $S' = (Q',U',(a',b'),C',\pi')$ the simulation of $\vvarCA$ by $\varCA$.
We define the set $C'' \subset \alp^{QQ' \times UU'}$ of macro-cells for the composition simulation by $C'' = \pi^{-1}(C')$, that is, a pattern $P \in \alp^{QQ' \times UU'}$ is in $C''$ if and only if $\pi(P) \in C'$.
The state function $\pi'' : \alp^{QQ' \times UU'} \to \vvaralp$ is the composition of $\pi'$ and $\pi$.
Define also $v = (a + Qa', b + Ub')$.
Then $(QQ',UU',v,C'',\pi'')$ is easily seen to be a simulation of $\vvarCA$ by $\CA$.
\end{proof}

The simulation $(QQ',UU',v,C'',\pi'')$ of $\vvarCA$ by $\CA$ constructed above is called the \emph{composition} of $S$ and $S'$, denoted $S' \circ S$.
The order of composition here is chosen to correspond to the composition of the state functions.

\begin{lemma}
\label{lem:CompProps}
Retain the assumptions and notation of Lemma~\ref{lem:Composition}.
If $S$ and $S'$ are rigid and connecting, then so is $S' \circ S$.
If $S$ is weakly rigid and connecting, and $S'$ is strongly rigid and strongly connecting, then $S' \circ S$ is again rigid and connecting.
\end{lemma}

\begin{proof}
First, if $S$ and $S'$ are rigid, we have $\pi''(\tilde \tra_\CA) = \pi'(\pi(\tilde \tra_\CA)) \subset \pi'(\tilde \tra_\varCA) \subset \tilde \tra_\vvarCA$, and then the composition is rigid.

Let $S$ and $S'$ be also connecting, and let $\eta \in \tilde \tra_\CA$.
If two cells $\pi'(\pi(\eta))^t_i$ and $\pi'(\pi(\eta))^s_j$ are adjacent, the bases $\pi(\eta)^{t U' + b'}_{i Q' + a'}$ and $\pi(\eta)^{s U' + b'}_{j Q' + a'}$ of the macro-cells $\pi(\eta)^{[t U', (t+1)U'-1]}_{[i Q', (i+1)Q'-1]}$ and $\pi(\eta)^{[s U', (s+1)U'-1]}_{[j Q', (j+1)Q'-1]}$ of $S'$ are directly connected, since $S'$ is connecting.
The directed path $(\pi(\eta)^{r_m}_{k_m})_{m=0}^M$ between the bases, where $(k_0,r_0) = (i Q' + a', t U' + b')$ and $(k_M,r_M) = (j Q' + a', s U' + b')$, consists of adjacent cells, and since $\pi(\eta) \in \tilde \tra_\varCA$ by rigidity and $S$ is connecting, the bases $\eta^{r_m U + b}_{k_m Q + a}$ and $\eta^{r_{m+1} U + b}_{k_{m+1} Q + a}$ of the macro-cells $\eta^{[r_m U,(r_m+1)U-1]}_{[k_m Q,(k_m+1)Q-1]}$ and $\eta^{[r_{m+1} U,(r_{m+1}+1)U-1]}_{[k_{m+1} Q,(k_{m+1}+1)Q-1]}$ of $S$ are directly connected for all $m \in [0,M-1]$.
But then the bases $\eta^{(t U' + b') U + b}_{(i Q' + a') Q + a}$ and $\eta^{(s U' + b') U + b}_{(j Q' + a') Q + a}$ of the macro-cells $\eta^{[t U U', (t+1)U U'-1]}_{[i Q Q', (i+1)Q Q'-1]}$ and $\eta^{[s U U', (s+1)U U'-1]}_{[j Q Q', (j+1)Q Q'-1]}$ of $S' \circ S$ are directly connected, and thus the composition is connecting.

The second claim is completely analogous, simply replace $\tilde \tra_\varCA$ by $\hat \tra_\varCA$ in the above.
\end{proof}

\section{Cellular Automata Represented by Programs}
\label{sec:Interpreter}

We now introduce a programming language for expressing CA rules.
As in \cite{Ga01}, we use it both for precisely describing complicated CA rules, and for encoding said rules into short bitstrings which can then be interpreted by Turing machines, in order to perform programmatic simulation.
The language describes automata whose states are binary vectors divided into fields, as defined in Section~\ref{sec:Defs}.
We do not rigorously define this language, but go through its relevant aspects using an example automaton, presented as a program in Algorithm~\ref{alg:Example}.
The program presented here is designed for illustrative purposes, and does not exhibit much interesting behavior.
As a convention, we only display line numbers in algorithms if we explicitly refer to them in the main text.

\begin{algorithm}[htp]
\caption{An example program of a cellular automaton.}\label{alg:Example}
\begin{algorithmic}[1]

\State $\mbf{num\ param}\ \mtt{Modulus} = 10$ \label{ln:Param}
\State $\mbf{num\ field}\ \mtt{Counter} \leq \mtt{Modulus} - 1$ \label{ln:Field}
\State
\Procedure{Proc}{$n$} \label{ln:Maindef}
	\If{$\mtt{Counter}^- = n$} \label{ln:Maincond}
		\State $\mtt{Counter} \gets \mtt{Counter} + 1 \bmod \mtt{Modulus}$ \label{ln:Mainset}
	\EndIf
\EndProcedure
\State
\If{$\mtt{Counter} \neq \mtt{Counter}^+$} \label{ln:Bodycond}
	\State \Call{Proc}{$\mtt{Counter}^+$} \label{ln:Maincall}
\Else
	\State \Call{Proc}{$0$} \label{ln:Maincall2}
\EndIf

\end{algorithmic}
\end{algorithm}

A program of the language consists of zero or more \emph{parameter and field definitions} in any order, followed by one or more \emph{procedure definitions} and finally the \emph{program body}.
On line \algref{alg:Example}{ln:Param}, an integer parameter called \texttt{Modulus} is defined.
Parameters should be viewed as constants, as their value is the same in all cells of the automaton, at all times.

Likewise, line \algref{alg:Example}{ln:Field} defines an integer field with maximum value $\mtt{Modulus} - 1 = 9$ (stored as a bitstring of length $\lceil \log_2 9 \rceil = 4$).
As opposed to parameters, fields make up the state of a cell, and their value typically changes as the cellular automaton is applied.
In this example, the state set of the automaton is $\{0, 1\}^4$, and it consists of this single field.
Fields and parameters can also contain Boolean truth values (stored as a single bit, with $0 = \mtt{False}$ and $1 = \mtt{True}$), bitstrings of some fixed length, or enumerations, which are simply integers with special labels.
For technical reasons that we explain later in this section, the length of a bitstring field must be given in unary notation.

Line \algref{alg:Example}{ln:Maindef} marks the beginning of a procedure definition.
The name of the procedure is \textsc{Proc}, and it takes one argument, $n$.
Line \algref{alg:Example}{ln:Maincond} marks the beginning of a conditional block, which is executed if the value of the implicitly defined parameter $\mtt{Counter}^-$, which contains the value of the $\mtt{Counter}$ field of the left neighbor, is exactly $n$.
Its counterpart $\mtt{Counter}^+$, containing the value of the $\mtt{Counter}$ field of the left neighbor, is referenced in line \algref{alg:Example}{ln:Bodycond}.
Only fields can be assigned new values; parameters are read-only, as are the fields of the neighboring cells.
On line \algref{alg:Example}{ln:Mainset}, the $\mtt{Counter}$ field of the cell is assigned a new value: it is incremented by one modulo the parameter $\mtt{Modulus}$.

The body of the program contains another conditional block, with two calls to the procedure \textsc{Proc} on lines \algref{alg:Example}{ln:Maincall} and \algref{alg:Example}{ln:Maincall2}, with the respective arguments $\mtt{Counter}^+$ and 0.
The transition rule of the CA consists of executing the program on every cell of the configuration.
In this case, the program encodes (in a somewhat convoluted way) the local rule
\[ (a,b,c) \mapsto \left\{ \begin{array}{cl}
	(b + 1) \bmod 10, & \mbox{if~} b = c = 0 \mbox{~or~} b \neq c = a \\
	b, & \mbox{otherwise,}
\end{array} \right. \]
where $a, b, c \in \{0, \ldots, 2^4-1 = 15\}$.

The programming language also contains a special keyword \texttt{This}.
It is handled as an implicitly defined bitstring parameter, whose value is \emph{the whole program}, using some suitable encoding.
In this example, the value of \texttt{This} would be the whole of Algorithm~\ref{alg:Example}, encoded as a bitstring.
It gives every program written in the language the ability to easily examine itself, and is mainly used for self-simulation.
We remark here that there are multiple ways of implementing programmatic self-simulation, but all of them include some `trick' that gives the cellular automaton access to its own definition.
We have chosen the approach of allowing our programs to explicitly refer to themselves, with the goal of being as easy to understand as possible.

This example contains most of the basic constructs of the programming language.
Procedure calls can appear in the bodies of other procedures, but recursion is explicitly prohibited, so there may not be a cycle of procedures with each calling the next.
The language can be assumed to contain any particular polynomial-time arithmetic and string manipulation functions we need, the latter of which can be used also on numerical fields.
It is clear that all cellular automata can be described by a program, if only by listing their entire transition table in a series of nested conditional blocks.
Conversely, we say a bitstring $p$ is a \emph{valid program} if it defines a cellular automaton.

We may now construct a Turing machine $\intp$, the \emph{interpreter}, that takes as input four binary words $r,s,t,p$ and outputs $\phi(r,s,t)$, where $\phi : \alp^3 \to \alp$ is the local function of the automaton $\CA : \alp^\Z \to \alp^\Z$ described by the valid program $p$.
In this case, we denote $\phi = \intp_p$, $\CA = \intp_p^\infty$ and $\size{\alp} = \size{p}$ (not to be confused with the length of the program $p$ as a bitstring, which is denoted by $|p|$).
Since we explicitly prohibited recursion, no form of looping is available, the length of each field is proportional to the length of its definition (here we need the unary encoding of bitstring lengths), and all functions run in polynomial time, it is clear that $\intp$ has time complexity $P^T_\intp(|p|, \size{p})$ and space complexity $P^S_\intp(|p|, \size{p})$, where $P^T_\intp$ and $P^S_\intp$ are some polynomials in two variables with integer coefficients.

Finally, we define CA transformations, which are functions that modify programmatically defined cellular automata.

\begin{definition}
Let $G : \N \times \{0, 1\}^* \to \{0, 1\}^*$ be a function.
We say that $G$ is a \emph{polynomial CA transformation}, if the following conditions hold.
\begin{itemize}
	\item $G(n, p)$ is computable in time polynomial in $n$ and $|p|$.
	\item If $p$ is a valid program, then so is $G(n, p)$ for all $n \in \N$.
	\item If $p$ is a valid program, then $\size{G(n, p)}$ is polynomial in $n$ and $\size{p}$.
\end{itemize}
If the automaton $\intp^\infty_{G(n,p)}$ only depends on $\intp^\infty_p$ for all $n \in \N$ and all valid programs $p$, we say $G$ is \emph{consistent}, and write $G(n,\CA)$ for the automaton $\intp^\infty_{G(n,p)}$, if $\CA$ is defined by the program $p$.
If $G$ and $G'$ are two polynomial CA transformations, so is the function $(n,p) \mapsto G(n, G'(n, p)))$, which is called their \emph{composition} and denoted $G \circ G'$.
\end{definition}

\section{Sparsification}
\label{sec:Sparse}

\subsection{The Sparsification Transformation}
\label{sec:SparseDef}

In this section, we define a specific polynomial CA transformation $G_s$ that performs \emph{sparsification}, and for that, let $\varCA = \intp^\infty_p : \varalp^\Z \to \varalp^\Z$ be a cellular automaton defined by the program $p \in \{0,1\}^*$, and let $n \in \N$ be arbitrary.
Intuitively, the CA defined by $G_s(n, p)$ behaves exactly like $\varCA$, except that all of its trajectories will eventually become \emph{sparse}, that is, they consist mostly of cells in the blank state.
In fact, the transformed automaton $\intp^\infty_{G_s(n, p)}$ simulates $\varCA$ with macro-cells that consist mostly of blank cells.

The transformation $G_s$ works as follows.
First, split the program $p$ into two parts $p_{\mathrm{def}}$ and $p_{\mathrm{body}}$, the first of which contains all the parameter, field and procedure definitions, and the second the main program body.
For a fixed $n \in \N$, the output of $G_s(n, p)$ is then Algorithm~\ref{alg:Sparsified}, along with the definitions of the procedures \textsc{Die}, \textsc{Blank}, \textsc{InputTriple}, \textsc{CopyRight} and \textsc{CopyLeft}.
If the original program happens to contain a parameter, field or procedure whose name clashes with the ones defined here, it will be renamed, together with all references to it.
It is clear that the conditions for being a consistent polynomial CA transformation are fulfilled.

\begin{algorithm}[htp]
\caption{The sparsified program $G_s(n, p)$, minus the new procedures.}\label{alg:Sparsified}
\begin{algorithmic}[1]

\State $\mbf{num\    param}\ \mtt{N} = n$ \Comment{The level of sparsification}
\State $\mbf{bool\   field}\ \live$ \Comment{Marks cells that are not blank}
\State $\mbf{enum\   field}\ \kind \in \{ {\Leftarrow}, {\Rightarrow}, {\Leftrightarrow} \}$ \Comment{Direction of movement}
\State $\mbf{num\    field}\ \cou \leq 2 \,\mtt{N}$ \Comment{Synchronizes the simulation}
\State $p_{\mathrm{def}}$ \Comment{The parameter, field and procedure definitions of $p$}
\State
\If{\Call{InputTriple}{}}
	\State $p_{\mathrm{body}}$ \label{ln:SparseCompute} \Comment{Execute the transition function defined by $p$}
	\If{\Call{Blank}{}}
		\Call{Die}{}
	\Else
		\State $\live \gets \mtt{True}$
		\State $\cou \gets 0$
		\State $\kind \gets {\Rightarrow}$
	\EndIf
\ElsIf{$\neg \live \wedge \neg \live^+ \wedge \kind^- = {\Rightarrow} \wedge \cou^- < 2\,\mtt{N}$}
	\State \Call{CopyLeft}{${\Rightarrow}$}
\ElsIf{$\kind = {\Leftarrow} \wedge \neg \live^+ \wedge \kind^- = {\Rightarrow} \wedge \cou^- = \cou = \mtt{N}$}
	\State \Call{CopyLeft}{${\Rightarrow}$} \label{ln:SparseCrossRight}
\ElsIf{$\neg \live \wedge \neg \live^- \wedge \kind^+ = {\Rightarrow} \wedge \cou^+ = 0$}
	\State \Call{CopyRight}{${\Leftarrow}$} \label{ln:SparseLeft}
\ElsIf{$\neg \live \wedge \neg \live^- \wedge \kind^+ = {\Leftarrow} \wedge 0 < \cou^+ < 2\,\mtt{N}$}
	\State \Call{CopyRight}{${\Leftarrow}$}
\ElsIf{$\kind = {\Rightarrow} \wedge \neg \live^- \wedge \kind^+ = {\Leftarrow} \wedge \cou^+ = \cou = \mtt{N}$}
	\State \Call{CopyRight}{${\Leftarrow}$} \label{ln:SparseCrossLeft}
\ElsIf{$\kind = {\Leftrightarrow} \wedge \neg \live^- \wedge \kind^+ = {\Leftarrow} \wedge \mtt{N} < \cou^+ = \cou < 2\,\mtt{N}$}
	\State \Call{CopyRight}{${\Leftarrow}$}
\ElsIf{$\neg \live \wedge \neg \live^- \wedge \kind^+ = {\Rightarrow} \wedge \cou^+ = \mtt{N}$}
	\State \Call{CopyRight}{${\Leftrightarrow}$} \label{ln:SparseReturn}
\ElsIf{$\neg \live \wedge \neg \live^- \wedge \kind^+ = {\Leftrightarrow} \wedge \mtt{N} < \cou^+ < 2\,\mtt{N}$}
	\State \Call{CopyRight}{${\Leftrightarrow}$}
\Else \,\Call{Die}{}
\EndIf

\end{algorithmic}
\end{algorithm}

We now explain the new fields and their functionality in the sparsified automaton.
Denote by $\alp$ the state set of $\CA = \intp^\infty_{G_s(n, p)}$.
The boolean field $\live$, called the \emph{live flag}, is used to indicate whether a cell is in the blank state $B_\alp$ ($\live = \mtt{False}$) or not ($\live = \mtt{True}$).
The field $\kind$ specifies the \emph{kind} of a non-blank cell.
As there are three possible values for the field, so are there three possible kinds: \emph{left moving}, \emph{right moving} and \emph{returning}, corresponding to the respective values ${\Leftarrow}$, ${\Rightarrow}$ and ${\Leftrightarrow}$.
The cells of different kinds travel back and forth (returning cells travel to the left), exchanging information with each other in order to simulate $\varCA$.
The field $\cou$, called the \emph{counter}, is used to synchronize the simulation.
Its maximum value is twice the parameter $\mtt{N} = n$.

The procedure definitions and program body of $G_s(n, p)$ are independent of $n$.
The procedure \textsc{Die} makes the current cell assume the blank state $B_\alp = 0^{\size{\alp}}$.
The procedure \textsc{Blank} returns $\mtt{True}$ if and only if the $\varalp$-state of the current cell is $B_\varalp$.
The procedure \textsc{InputTriple} returns $\mtt{True}$ if and only if the cells in the neighborhood, if they are live, have counter value $2\,\mtt{N}$, their kinds (if live) are ${\Rightarrow}$, ${\Leftrightarrow}$ and ${\Leftarrow}$, and at least one of them is live.
Finally, \textsc{CopyRight}, when given an argument $k \in \{ {\Leftarrow}, {\Rightarrow}, {\Leftrightarrow} \}$, makes the current cell live, sets its kind to $k$, copies the counter value and $\varalp$-state of the right neighbor, and then increments the counter value by one.
The behavior of \textsc{CopyLeft} is symmetric.

A correct simulation of $\CA$ proceeds as follows.
In the beginning of what is called a \emph{cycle} of $2n + 1$ steps, a right moving cell with counter value $0$ is placed at the coordinate $i(2n+1)$ for every $i \in \N$, with blank cells between them.
In general, at each timestep $t \in \N$, every live cell will have counter value $t \bmod 2\,\mtt{N}+1$.
The right moving cells create a left moving cell on their left, which has the same $\alp$-state as they do (line~\algref{alg:Sparsified}{ln:SparseLeft}).
These cells move in their respective directions until their counter values reach $2\,\mtt{N}$, and the right and left cells originating from neighboring segments may cross at $t = \mtt{N}$ (lines~\algref{alg:Sparsified}{ln:SparseCrossRight} and~\algref{alg:Sparsified}{ln:SparseCrossLeft}).
At that time, the right moving cells also create a returning cell on their left (line~\algref{alg:Sparsified}{ln:SparseReturn}), which starts moving left along the left moving cell.
At $t = 2\,\mtt{N}$, each coordinate $i(2n+1)$ contains a returning cell with the $\alp$-state it had at the beginning of the cycle, and is surrounded by a right-moving and left-moving cell, containing the $\alp$-states of its left and right neighbors, respectively.
The procedure \textsc{InputTriple} returns $\mtt{True}$ for the returning cells, and these three-cell groups are used to calculate new $\varalp$-states for the cells, as per the rule of $\varCA$ (line~\algref{alg:Sparsified}{ln:SparseCompute}).
Note that if some of the cells in the neighborhood are blank, the computation on line~\algref{alg:Sparsified}{ln:SparseCompute} will be carried out as if they were live cells with $\varalp$-state $B_\varalp$.
Also, if the $\varalp$-state of the resulting cell is $B_\varalp$, it will die.
Thus the blank state $B_\varalp$ is simulated by the absence of a live cell.

See Figure~\ref{fig:SparsifiedXOR} for a sparsification of the three-neighbor XOR automaton, which should clarify the above explanation.

\input{SparsifiedXOR2}

We also slightly generalize the sparsification construction by defining partial sparsifications.

\begin{definition}
A \emph{partial sparsification transformation} is a polynomial CA transformation $G_s^N$, where $N \subseteq \N$ is decidable in polynomial time, defined by
\[ G_s^N(n, p) = \left\{ \begin{array}{cl} G_s(n, p), & \mbox{if $n \in N$} \\ p, & \mbox{otherwise.} \end{array} \right. \]
\end{definition}

Clearly, each $G^N_s$ is also consistent.
The set $N$ is called the \emph{effective set} of $G^N_s$.
Partial sparsification transformations are used in the computability results of Section~\ref{sec:Computability}, where classes of cellular automata are separated by whether $N$ is finite or infinite.

\subsection{Properties of Sparsification}
\label{sec:SparseProp}

We now prove some basic properties of the sparsification transformation, keeping the notation of Section~\ref{sec:SparseDef}.
For the rest of this section, fix $n \in \N$ and denote $M = 2n+1$.
First, there is the claim that the sparsified automaton actually simulates the original.

\begin{lemma}
\label{lem:SparseSim}
The CA $\CA = \intp^\infty_{G_s(n,p)}$ $(M,M)$-simulates $\varCA$.
\end{lemma}

\begin{proof}
We define the set $C \subset \alp^{M^2}$ of macro-cells as the set of those rectangular patterns $P \in \alp^{M^2}$ that satisfy $\live(P^0_0) = \mtt{True}$, $\kind(P^0_0) = \mtt{Right}$ and $\cou(P^0_0) = 0$.
These are the black cells in Figure~\ref{fig:SparsifiedXOR}.
For such a pattern $P$, define $\pi(P)$ simply as the $\alp$-state of $P^0_0$.
If $P \notin C$, we of course define $\pi(P) = B_\alp$.
We now claim that $(M,M,(0,0),C,\pi)$ is the desired simulation.

Let thus $\eta \in \tra^+_\varCA$ be a trajectory of $\varCA$, and define $\xi \in \tra^+_\CA$ by $\xi^0_{[Qi,Q(i+1)-1]} = c^t_i B_\alp^{2n}$, where $c^t_i \in \alp$ is either a live right moving cell with counter value $0$ and $\varalp$-state $\eta^0_i$ if $\eta^t_i \neq B_\varalp$, or $B_\alp$ if $\eta^t_i = B_\varalp$.
Every other row of $\xi$ is of course the $\CA$-image of the previous row.
Then we have $\pi(\xi)^0 = \eta^0$ by definition of $\pi$.
The general case of $\pi(\xi)^t = \eta^t$ for all $t \in \N$ is easily verified by induction, using the rules of $\CA$ and the explanation of the simulation in the previous subsection.
\end{proof}

The simulation $(M,M,(0,0),C,\pi)$ defined above is called the \emph{level-$n$ sparse simulation of $\CA$}.
In general, the sparse simulation is not rigid, but we can quite easily show the following.

\begin{lemma}
\label{lem:SparseRigidity}
The sparse simulation of $\CA$ by $\varCA = \intp^\infty_{G_s(n,p)}$ is weakly rigid.
\end{lemma}

\begin{proof}
Let $\eta \in \tilde \tra_\CA$ be arbitrary.
It suffices to prove that we have $\pi(\eta)^1_0 = \intp_p(c_{-1}, c_0, c_1)$ in the case that $\pi(\eta)^1_0 \neq B_\alp$, for some choice of $c_\delta \in \{\pi(\eta)^0_\delta, B_\varalp\}$ for all $\delta \in \{-1,0,1\}$.
Denote thus $P = \eta^{[M,2M-1]}_{[0,M-1]} \in C$.

Consider the cells $\eta^{M-1}_\delta$ for $\delta \in \{-1,0,1\}$ which lie directly below the base $P^0_0$ in $\eta$.
Since $P^0_0$ is live, its $\varalp$-state is obtained by applying $\intp_p$, the local rule of $\varCA$, to the $\varalp$-states of these three cells.
If $\eta^{M-1}_{-1}$ is live, then we can prove by induction, using the rules of $\CA$, that each $\eta^t_{t-M}$ for $t \in [0,M-1]$ is a right moving cell, and they all have the same $\varalp$-state, say $c \in \varalp$.
But this implies that $\eta^{[0,M-1]}_{[-M,-1]}$ is a macro-cell in state $c$, and then $\pi(\eta)^0_{-1} = c$.
Analogously, we can show for all $\delta \in \{-1,0,1\}$ that if $\eta^{M-1}_\delta$ is live, then $\pi(\eta)^0_\delta$ equals its $\varalp$-state.
This finishes the proof.
\end{proof}

\begin{lemma}
\label{lem:SparseConnect}
The sparse simulation of $\CA$ by $\varCA = \intp^\infty_{G_s(n,p)}$ is connecting.
\end{lemma}

\begin{proof}
Let $\eta \in \tilde \tra_\CA$, and suppose $\pi(\eta)^t_i$ and $\pi(\eta)^{t-1}_{i+\delta}$ are adjacent for some $i, t \in \Z$ and $\delta \in \{-1,0,1\}$.
We assume that $\delta = -1$, as the other cases are similar.
We now have $\eta^{[tM,(t+1)M-1]}_{[iM,(i+1)M-1]}, \eta^{[(t-1)M,tM-1]}_{[(i-1)M,iM-1]} \in C$, and by the definition of $C$, $\eta^{tM}_{iM}$ and $\eta^{(t-1)M}_{(i-1)M}$ are both right moving cells with counter value $0$.
Since $\pi(\eta)^t_i$ is $(-1)$-demanding, so is $\eta^{tM}_{iM}$, and thus $\eta^{tM-1}_{iM-1}$ is a right moving cell with counter value $2n$.
Right moving cells with nonzero counter value are all $(-1)$-demanding, and thus $\eta^{tM-\ell}_{iM-\ell}$ is a $(-1)$-demanding cell for all $\ell \in [0,2n]$.
This shows that the bases of the two macro-cells are directly connected.
\end{proof}

Next, we show that every live cell in a nondeterministic trajectory of $\CA$ originates from a macro-cell in the following sense.

\begin{lemma}
\label{lem:SparseParents}
Let $\eta \in \tilde \tra_\CA$, and let $\eta^t_i$ be a live cell.
Then $\eta^{t-s}_j$ is the base of a macro-cell for some $s \in [0,2n]$ and $j \in [-n,n]$, and is directly connected to $\eta^t_i$.
\end{lemma}

\begin{proof}
First, if $\cou(\eta^t_i) = 0$, then $\eta^t_i$ is necessarily the base of a macro-cell, and the claim holds.
Otherwise, $\eta^t_i$ is $\delta$-demanding for some $\delta \in \{-1,0,1\}$, and the claim follows by induction.
\end{proof}

It can be easily proved by induction that if $\eta \in \tilde \tra_\CA$, and two cells $\eta^t_i$ and $\eta^s_j$ are weakly connected and have the same kind and counter value, then $|i-j|$ and $|t-s|$ are both divisible by $M_n$.
The following two lemmas directly follow from this observation, and the former is the main reason for defining $G_s$ in the first place.

\begin{lemma}
\label{lem:Sparsity}
Let $\eta \in \tilde \tra_\CA$, and let $(i,t) \in \Z^2$.
Let $r \in \N$, and denote
\[ D(r) = \{ s \in [0,r-1] \;|\; \eta^t_i \mbox{~is connected to~} \eta^s_0\}. \]
Then $|D(r)| \leq \frac{3r}{n} + 3$.
\end{lemma}

\begin{proof}
Simply note that if two cells on the same column of $\eta$ are connected and have the same kind, then they also have the same counter value, and use the above observation.
\end{proof}

\begin{lemma}
\label{lem:SparseEvenCells}
Let $\eta \in \tilde \tra_\CA$, and let $(i,t), (j,s) \in \Z^2$ be such that $\eta^{[t,t+M_n-1]}_{[i,i+M_n-1]}$ and $\eta^{[s,s+M_n-1]}_{[j,j+M_n-1]}$ are macro-cells, and their bases are weakly connected.
Then $|i-j|$ and $|t-s|$ are both divisible by $M_n$.
\end{lemma}

\begin{remark}
In the following sections, up to and including the proof of the main theorem, we will not refer to the rules of any sparsified automaton, but only to the lemmas proved here, and the dimensions and base coordinate of the sparse simulations.
Indeed, the exact implementation of the sparsification transformation is largely irrelevant.
\end{remark}

\section{Universal Programmatic Simulation}
\label{sec:UnivSim}

In this section, we construct a cellular automaton that simulates another automaton in a specific but highly modular way.
The details of the construction are mainly borrowed from \cite{Ga01}.
However, there is an important conceptual difference between the error-correcting, self-organizing simulator of G\'acs, and the following construction.
Namely, whereas the G\'acs automaton is designed to detect and fix small errors in the simulation, our universal simulator is constructed so that every error will propagate and wipe out as much of the simulation structure as possible, replacing it by an area of blank cells.
For this reason, it is also much simpler.

\subsection{The Universal Simulator Automaton}
\label{sec:UnivConstr}

We are now ready to present the construction for simulating a single arbitrary cellular automaton.
For that, let $p$ be a valid program that encodes the CA $\intp_p^\infty$ on the alphabet $\varalp$.
We construct a cellular automaton $\CA : \alp^\Z \to \alp^\Z$ that simulates $\intp_p^\infty$ by presenting a valid program that defines $\CA$.
As with the sparsification transformation, there are many essentially similar ways of achieving universal programmatic simulation, and many details in our construction of $\CA$ are somewhat arbitrary.

\begin{algorithm}[htp]
\caption{The parameters and fields of $\Psi$, with comments indicating their purpose.}
\label{alg:ParamsFields}
\begin{algorithmic}

\State $\mbf{num\ 		param}\ \mtt{CSize} = Q$ \Comment{Size of colonies, width of macro-cells}
\State $\mbf{num\ 		param}\ \mtt{WPeriod} = U$ \Comment{Length of work period, height of macro-cells}
\State $\mbf{string\ 	param}\ \mtt{SimProg} = p$ \Comment{The program to be simulated}
\State
\State $\mbf{bool\ 		field}\ \live$ \Comment{Marks a cell that is not blank}
\State $\mbf{num\ 		field}\ \addr \leq \mtt{CSize} - 1$ \Comment{Position of a cell in its colony}
\State $\mbf{num\ 		field}\ \age  \leq \mtt{WPeriod} - 1$ \Comment{Synchronizes the cells of a colony}
\State $\mbf{bool\ 		field}\ \simu$ \Comment{Stores one bit of the state of $\intp_p^\infty$}
\State $\mbf{enum\ 		field}\ \work \in (Q_\mathrm{A} \cup \{\#\}) \times \Sigma_\mathrm{A}$ \Comment{Contains the agent and/or its data}
\State $\mbf{bool\ 		field}\ \prog$ \Comment{Stores one bit of $\mtt{SimProg}$ for the agent}
\State $\mbf{bool\ 		field}\ \lmail$ \Comment{Retrieves the state of the left neighbor}
\State $\mbf{bool\ 		field}\ \rmail$ \Comment{Retrieves the state of the right neighbor}
\State $\mbf{bool\ 		field}\ \out$ \Comment{Stores one bit of the output of the agent}
\State $\mbf{bool\ 		field}\ \lbor$ \Comment{Marks a left border cell}
\State $\mbf{bool\ 		field}\ \rbor$ \Comment{Marks a right border cell}
\State $\mbf{bool\ 		field}\ \doom$ \Comment{Marks a cell that dies at the end of the work period}

\end{algorithmic}
\end{algorithm}

\begin{algorithm}[htp]
\caption{The program body of $\CA$, and the procedure \textsc{Die}.}\label{alg:MainBody}
\begin{algorithmic}[1]

\Procedure{Die}{}
	\State $\live \gets \mtt{False}$
	\State $\addr, \age \gets 0$
	\State $\simu, \prog, \lmail, \rmail, \out \gets 0$
	\State $\work \gets (\#, B_\mathrm{A})$
	\State $\lbor, \rbor, \doom \gets \mtt{False}$
\EndProcedure

\If{$\live$}
	\If{$\neg$\Call{Valid}{}} \label{ln:Valid}
		\Call{Die}{} \label{ln:Invalid} \Comment{Destroy locally incorrect structure}
	\Else
		\If{$0 \leq \age < \mtt{CSize}$} \label{ln:Retrieve}
			\State $\lbor, \rbor \gets \mtt{False}$ \label{ln:NotCreative} \Comment{Remove unnecessary border flags}
			\State $\lmail \gets \lmail^-$ \Comment{Transfer information between colonies}
			\State $\rmail \gets \rmail^+$
		\EndIf
		\If{$\age = \mtt{CSize}$}
			\State $\doom \gets \mtt{False}$ \Comment{Remove existing doom flags}
			\If{$\lbor \wedge \live^-$} \Comment{Remove unnecessary border flags}
				\State $\lbor \gets \mtt{False}$
			\ElsIf{$\rbor \wedge \live^+$}
				\State $\rbor \gets \mtt{False}$
			\EndIf
			\If{$\addr = 0$} \Comment{Initialize computation}
				\State $\work \gets (q_\mathrm{init}, B_\mathrm{A})$ \Comment{Initial state and blank tape letter}
			\Else
				\State $\work \gets (\#, B_\mathrm{A})$ \Comment{No head and blank tape letter}
			\EndIf
			\State $\prog \gets \mtt{SimProg}_{\addr}$ \label{ln:StoreProg} \Comment{Returns $0$ if $\addr \geq |\mtt{SimProg}|$}
		\EndIf
		\If{$\mtt{CSize} < \age < \mtt{WPeriod} - 1$} \Comment{Compute new state}
			\State \Call{Compute}{}
		\EndIf
		\If{$\age = \mtt{WPeriod} - 1$}
			\If{$\addr = 0 \wedge \doom^-$}  \Comment{Prepare against dying neighbors}
				\State $\lbor \gets \mtt{True}$ \label{ln:Protect1}
			\ElsIf{$\addr = \mtt{CSize} - 1 \wedge \doom^+$}
				\State $\rbor \gets \mtt{True}$ \label{ln:Protect2}
			\EndIf
			\If{$\doom$}
				\Call{Die}{} \label{ln:Doomed} \Comment{Kill a doomed colony}
			\Else
				\State $\simu, \rmail, \lmail \gets \out$ \Comment{Assume new state}
			\EndIf
		\EndIf
		\State $\age \gets \age + 1 \bmod \mtt{WPeriod}$
	\EndIf
\Else
	\ \Call{Birth}{} \Comment{Become a live cell if necessary}
\EndIf

\end{algorithmic}
\end{algorithm}

\begin{algorithm}[htp]
\caption{The rules of local validity of a live cell in the universal simulation.}\label{alg:Validity}
\begin{algorithmic}[1]

\Procedure{Valid}{}
	\If{$(\live^- = \lbor) \wedge \neg(\lbor \wedge \rbor^- \wedge \frac{\mtt{CSize}}{\age} \in \{1, 2\})$} \label{ln:BlankInvalidL}
		\State \Return{$\mtt{False}$}
	\ElsIf{$(\live^+ = \rbor) \wedge \neg(\rbor \wedge \lbor^+ \wedge \frac{\mtt{CSize}}{\age} \in \{1, 2\})$} \label{ln:BlankInvalidR}
		\State \Return{$\mtt{False}$}
	\ElsIf{$\lbor^+ \vee \rbor^-$} \label{ln:BorderInvalid}
		\State \Return{$\mtt{False}$}
	\ElsIf{$\live^- \wedge ( \addr^- + 1 \not\equiv \addr \bmod \mtt{CSize} \vee \age^- \neq \age )$}
		\State \Return{$\mtt{False}$}
	\ElsIf{$\live^+ \wedge ( \addr \not\equiv \addr^+ - 1 \bmod \mtt{CSize} \vee \age \neq \age^+ )$}
		\State \Return{$\mtt{False}$}
	\ElsIf{$\lbor \wedge (\age \geq \mtt{CSize}) \wedge (\addr \neq 0)$}
		\State \Return{$\mtt{False}$}
	\ElsIf{$\rbor \wedge (\age \geq \mtt{CSize}) \wedge (\addr \neq \mtt{CSize} - 1)$}
		\State \Return{$\mtt{False}$}
	\Else
		\State \Return{$\mtt{True}$}
	\EndIf
\EndProcedure

\end{algorithmic}
\end{algorithm}

The parameters and fields of $\CA$ are defined in Algorithm~\ref{alg:ParamsFields}, and its main body is shown in Algorithm~\ref{alg:MainBody}.
The first field, $\live$, is called the \emph{live flag}, and it has the same purpose as in the sparsification transformation.
The procedure \textsc{Die} is also analogous to its earlier counterpart, making the current cell assume the state $B_\alp$.

The fields $\addr$ and $\age$, respectively called \emph{address} and \emph{age}, are used to synchronize the behavior of distinct cells.
Their maximum values are given by the parameters $Q$, called the \emph{colony size}, and $U$, the \emph{work period}, respectively.
The parameter $Q$ is assumed to be divisible by 4.
The age field is analogous to the $\cou$ field of the sparsification, while the address field serves the analogous purpose in the horizontal direction.
The procedure \textsc{Valid}, shown in Algorithm~\ref{alg:Validity}, is used to determine whether the local structure of the simulation is correct.
In particular, for cells with $\lbor = \rbor = \mtt{False}$, it returns $\mtt{True}$ only if $\age^- = \age = \age^+$, $\addr^- + 1 \equiv \addr \equiv \addr^+ - 1 \bmod Q$ and $\live^- = \live = \live^+ = \mtt{True}$.
Thus, in `normal' conditions, the blank state spreads through live cells, though the fields $\lbor$ and $\rbor$ can change this behavior.
Also, the address $\addr(x_i)$ of a cell $i \in \Z$ is normally $\addr(x_{i+1})$ decremented by one modulo $Q$, and adjacent cells have the same $\age$ value.
Following \cite{Ga01}, a subword $w = x_{[i,i+Q-1]}$ of length $Q$ is called a \emph{colony}, if $\addr(w_j) = j$ for all $j \in [0,Q-1]$.
The main idea of the construction is that each colony corresponds to a cell of the simulated automaton $\intp^\infty_p$, and explicitly computes a new state for itself every $U$ steps, timed by the age field.
The numbers $Q$ and $U$ are thus the dimensions of the simulation.

The next six fields (from $\simu$ to $\out$) are used to programmatically simulate the automaton $\intp^\infty_p$.
First, the Boolean field $\simu$, called the \emph{simulation bit}, is used to store the simulated $\varalp$-state of a colony $w \in \alp^Q$ as the bitstring $\simu(w_0) \simu(w_1) \cdots \simu(w_{\size{\varalp}-1}) \in \{0,1\}^{\size{\varalp}} = \varalp$.

The work period is divided into two phases: the \emph{retrieval phase} and the \emph{computation phase}.
The retrieval phase covers the first $Q$ steps of the work period, and is governed by the rules of the $\lmail$ and $\rmail$ fields, called the \emph{left and right mail fields}, respectively.
This phase is entered in line~\algref{alg:MainBody}{ln:Retrieve}, and it consists of moving the contents of the right and left mail fields $Q$ steps to the left and right, respectively.
In a correct simulation, both mail fields are equal to the simulation bit in all cells at the beginning of the work period, and at the end of the retrieval phase, the left (right) mail fields of a colony then contain exactly the simulation bits of its left (right, respectively) neighbor.

In the computation phase, a Turing machine head, called the \emph{agent}, is simulated on the cells of the colony, using the $\work$ fields, called the \emph{workspace fields}.
These fields contain elements of $(Q_\mathrm{A} \cup \{\#\}) \times \Sigma_\mathrm{A}$, where $Q_\mathrm{A}$ is the state set of the agent, and $\Sigma_\mathrm{A} \times \{0, 1\}^4$ is its tape alphabet.
At the beginning of the phase, the program $p$ is stored in the $\prog$ bits (the \emph{program bits}) of the first $|p|$ cells of every colony $w \in \alp^Q$, so that $\prog(w_0) \prog(w_1) \cdots \prog(w_{|p|-1}) = p$ (line~\algref{alg:MainBody}{ln:StoreProg}, where $p_i = 0$ if $i \geq |p|$).
At age $Q$, the workspace fields of every cell of the colony except the leftmost one are cleared, and the agent is placed on the leftmost cell in its initial state.
The procedure \textsc{Compute} performs one step of computation of the agent.
During the phase, the agent simulates the interpreter $\intp$ on the contents of the $\lmail$, $\simu$, $\rmail$ and $\prog$ fields of the colony, considering them as binary read-only tapes, and using the $\Sigma_\mathrm{A}$-components of the workspace fields as a read-and-write tape.
After at most $P^T_\intp(|p|, \size{p})$ steps, it writes the result in the $\out$ fields (the \emph{output bits}) of the colony.
If there is not enough room to perform the computation, the procedure \textsc{Die} is called.

The three remaining fields, $\doom$, $\lbor$ and $\rbor$ are called the \emph{doom flag}, \emph{left border flag} and \emph{right border flag}, respectively, and their purpose is to make $\CA$ simulate the blank state of $\intp^\infty_p$ using blank states, as in the sparsification transformation.
The `default value' of these fields is $\mtt{False}$, and a true value indicates some special circumstance.

In order to simulate blank states using blank states, some colonies have to be killed, and we choose to destroy those that are about to enter the state $B_\varalp$.
After completing the simulation of $\intp^\infty_p$ in the computation phase, the agent checks whether the new state of the colony would be $B_\varalp$, and if so, it sets the doom flag of every cell in the colony to $\mtt{True}$.
Otherwise, the flags remain $\mtt{False}$.
On line~\algref{alg:MainBody}{ln:Doomed}, it is stated that a cell with age $U-1$ and doom flag $\mtt{True}$ will immediately die.
Thus every colony trying to switch to the blank state will be wiped out at the end of its work period, being replaced by a segment $B_\alp^Q$ of blank cells.

Now, we also need a way to reclaim the blank segments, and this is where the border flags come in.
Cells with age less than $Q$ and left (right) border flag $\mtt{True}$ are called \emph{left (right, respectively) creative}.
The purpose of creative cells is to allow the simulation to create a new colony in the place of a destroyed one.
In the retrieval phase, a colony whose right neighbor has been replaced with blank cells will `push' its simulation bits into the blank segment, rebuilding the neighbor (who still simulates the blank state) from the left one cell at a time.
At any given time, the rightmost cell created in the new colony is right creative, but the others are not.
Thus the creative status slides to the right at speed 1, or in other words, the only right creative cells in the unfinished colony with address $a \in [0, Q-1]$ has age $a+1$.
To achieve this, a right creative cell whose right neighbor is blank is considered valid (and is not killed on line~\algref{alg:MainBody}{ln:Invalid}), provided that it has a consistent left neighbor and the above relation holds for its address and age.
Also, according to procedure \textsc{Valid}, a creative cell must have a blank neighbor in order to be structurally valid and survive, unless its age is exactly $\frac12 Q$ (it is necessary to allow this special case, since a colony may be created from both sides simultaneously).
The analogous rules hold for left creative cells.

This process of constructing a new colony is called \emph{colony creation}, and it is governed by the procedure \textsc{Birth} (Algorithm~\ref{alg:Birth}), which is only called on blank cells, and line~\algref{alg:MainBody}{ln:NotCreative} of the main program.
By the procedure \textsc{Birth}, if a blank cell has a left (right) neighbor which is right (left) creative (see lines~\algref{alg:Birth}{ln:BirthCond1} and~\algref{alg:Birth}{ln:BirthCond2}), then it will become live, in a state that is consistent with that of the neighbor, and with $\simu$-bit and left (right, respectively) mail field $0$.
Furthermore, the new cell will get a $\mtt{True}$ blocking flag itself.
According to line~\algref{alg:MainBody}{ln:NotCreative}, a creative cell always ceases to be creative on the next step, which causes only the bordermost cell of the creation process to be creative at any given time.
These rules cause a new colony to be created during the retrieval phase next to an existing colony surrounded by blank cells on either side.
The simulated state of the new colony is $B_\varalp$, and its mail fields contain the simulation bits of the neighboring colonies that created it.
To initialize the creation process, when a colony of doomed cells is about to be replaced by $B_\alp^Q$, the border cells of the neighboring colonies become blocking to protect them from the resulting blank cells (lines~\algref{alg:MainBody}{ln:Protect1} and~\algref{alg:MainBody}{ln:Protect2}).

If several neighboring colonies have been killed at the same time, a newly created colony may end up next to another blank segment, and needs to protect itself from it.
This is also achieved with the border flags.
Cells with address $0$ ($Q-1$), age at least $Q$ and left (right) border flag $\mtt{True}$ are called \emph{left (right, respectively) blocking}.
As with creative cells, a left (right) blocking cell whose left (right) neighbor is in state $B_\alp$ is considered valid by the procedure \textsc{Valid}, provided that its address and age are consistent with those of its right (left, respectively) neighbor.
Up to one exception, these are the only cells with age at least $Q$ and border flag $\mtt{True}$ that are considered valid.
The one exception is given by a left (right) blocking cell of age $Q$ surrounded by live cells with consistent addresses and ages, the leftmost (rightmost) of which is also right (left, respectively) blocking.
Such cells will also not die, but simply become non-blocking.
This situation happens when two neighboring colonies are created simultaneously.
At the end of the computation phase, blocking cells with age $U - 1$ that are not doomed become creative, and start the creation process of yet other colonies.

\begin{algorithm}[htp]
\caption{The rules for the creation of live cells cells.}\label{alg:Birth}
\begin{algorithmic}[1]

\Procedure{Birth}{} \Comment{Executed if and only if $\neg \live$ holds}
	\State \Call{Die}{} \Comment{Guarantees that dead cells are blank}
	\If{$\lbor^+ \wedge (\age^+ = \mtt{CSize} - \addr^+ \bmod \mtt{CSize})$} \label{ln:BirthCond1}
		\State $\live \gets \mtt{True}$
		\State $\addr \gets \addr^+ - 1 \bmod \mtt{CSize}$
		\State $\age \gets \age^+ + 1$
		\State $\rmail \gets \rmail^+$
		\State $\lbor \gets \mtt{True}$
	\ElsIf{$\rbor^- \wedge (\age^- = \addr^- + 1 \bmod \mtt{CSize})$} \label{ln:BirthCond2}
		\State $\live \gets \mtt{True}$
		\State $\addr \gets \addr^- + 1 \bmod \mtt{CSize}$
		\State $\age \gets \age^- + 1$
		\State $\lmail \gets \lmail^-$
		\State $\rbor \gets \mtt{True}$
	\EndIf
\EndProcedure

\end{algorithmic}
\end{algorithm}

We now make some useful definitions and observations about the local rule of the universal simulator automaton.
They can be verified by inspecting the program of $\CA$.

\begin{definition}
Let $P$ be a pattern over $\alp$, finite or infinite.
Two cells $P^t_i$ and $P^s_j$ of $P$ are \emph{consistent}, if they are live and satisfy $\addr(P^t_i) - \addr(P^s_j) \equiv i - j \bmod Q$ and $\age(P^t_i) - \age(P^s_j) \equiv t - s \bmod U$.
If $P \in \alp^{Q \times U}$, we say that the cell $P^t_i$ is \emph{consistent with $P$}, if it is live and satisfies $\addr(P^t_i) = i$ and $\age(P^i_t) = t$.
If $P$ is a subpattern of a larger pattern, we extend this terminology to cells outside it in the natural way.

A live cell with age in $[1,Q]$ and left (right) border flag $\mtt{True}$ is called \emph{left (right) co-creative}.
\end{definition}

Another word for co-creative cells could have been `created', since co-creative cells are created by creative cells during a correct simulation, but since the notion is technical, we chose a more distinguishable term.
Note that a blocking cell with age $Q$ is co-creative.

Note that in the following observations, we are dealing with a trajectory of the nondeterministic CA $\hat \CA$, which may regard any cell in a neighborhood as $\B_\alp$, independently of the other cells and the other neighborhoods.

\begin{observation}
\label{obs:Creative}
Let $\eta \in \hat \tra_\CA$, let $(i,t) \in \Z^2$, and denote $c = \eta^t_i$.
\begin{itemize}
\item If $c$ is left creative, then $\age(c) = Q - \addr(c) \bmod Q$.
If $d = \eta^{t+1}_{i-1}$ is also consistent with $c$, then either $d$ is left co-creative, or $\age(c) = \frac12 Q$ and $\eta^t_{i-1}$ is consistent with $c$ and right co-creative.
\item If $c$ is left co-creative, then $\age(c) = Q - \addr(c)$, and $\eta^{t-1}_{i+1}$ is consistent with $c$ and left creative.
\end{itemize}
The symmetric claim, with the left and right notions switched and $Q - \addr(c)$ replaced by $\addr(c) + 1$, also holds.
\end{observation}

\begin{observation}
\label{obs:Inconsistent}
Let $\eta \in \hat \tra_\CA$ and $P = \eta^{[0,U-1]}_{[0,Q-1]}$.
Let $\delta \in \{-1, 0, 1\}$, and let $(i,t) \in [0,Q-1] \times [0,U-1]$ be such that exactly one of $c = \eta^{t+1}_{i+\delta}$ and $\eta^t_i$ is consistent with $P$.
Then at least one of the following conditions holds:
\begin{itemize}
\item $c$ is blank;
\item $c$ is co-creative, and $t \in [0, Q-1]$;
\item $\delta = -1$, $c$ is right blocking and $i+\delta \equiv Q-1 \bmod Q$;
\item $\delta = -1$, $\eta^t_{i+\delta}$ is right co-creative, and $t \in [1, Q]$;
\item $\delta = 1$, $c$ is left blocking and $i+\delta \equiv 0 \bmod Q$; or
\item $\delta = 1$, $\eta^t_{i+\delta}$ is left co-creative, and $t \in [1, Q]$.
\end{itemize}
\end{observation}

Finally, we prove that the automaton $\CA$ actually simulates $\intp^\infty_p$ with the parameters $Q$ and $U$, provided that they are sufficiently large.
Recall that $\size{p}$ is defined as $\size{\varalp}$.

\begin{lemma}
\label{lem:UnivSimu}
Define $C \subset \alp^{Q \times U}$ as the set of rectangular patterns $P \in \alp^{Q \times U}$ whose $2Q$'th row from the top satisfies $\live(P^{U-2Q}_i) = \mtt{True}$, $\age(P^{U-2Q}_i) = U - 2Q$ and $\addr(P^{U-2Q}_i) = i$ for all $i \in [0, Q-1]$.
For all $P \in C$, define $\pi(P) = \simu(P^{U-2Q}_0) \simu(P^{U-2Q}_1) \cdots \simu(P^{U-2Q}_{\size{p} - 1}) \in \varalp$, the word formed by the $\simu$-bits of the first $\size{p}$ cells of said row.
For $P \in \alp^{Q \times U} \setminus C$, define $\pi(P) = B_\varalp$.
If the values $Q$ and $U$ satisfy
\begin{equation}
\label{eq:QBound}
Q \geq P^S_\intp(|p|, \size{p})
\end{equation}
and
\begin{equation}
\label{eq:UBound}
U \geq 6Q + P^T_\intp(|p|, \size{p}),
\end{equation}
then $(Q, U, (0,U-2Q), C, \pi)$ is a simulation of $\intp^\infty_p$ by $\CA$.
\end{lemma}

\begin{proof}
Denote $\varCA = \intp^\infty_p$.
The condition $\pi(\alp^{Q \times U} \setminus C) = B_\varalp$ holds by definition, so we only need to show $\tra^+_\varCA \subset \pi(\tra^+_\CA)$.
For that, let $\eta \in \tra^+_\varCA$ be arbitrary.
Define the one-dimensional configuration $x \in \alp^\Z$ as follows.
For each $i \in \Z$ and $j \in [0,Q-1]$, define the cell $s = x_{iQ + j} \in \alp$ by
\[ \begin{array}{rclrclrcl}
\live(s) &=& \mtt{True},  & \out(s)  &=& 0,                       & \simu(s)  &=& (\eta^0_i)_j, \\
\addr(s) &=& j,           & \work(s) &=& (\#, \Sigma_\mathrm{A}), & \lmail(s) &=& \simu(s), \\
\age(s)  &=& 0,           & \prog(s) &=& 0,                       & \lmail(s) &=& \simu(s), \\
\lbor(s) &=& \mtt{False}, & \rbor(s) &=& \mtt{False},             & \doom(s)  &=& \mtt{False}. \\
\end{array} \]

Now, each segment $x_{[iQ,(i+1)Q-1]} \in \alp^Q$ is a colony with age $0$ and $\varalp$-state $\eta^0_i$.
Since the mail fields of these colonies have the same contents as their simulation fields, the information retrieved by their neighbors is correct, and after $U$ steps, every colony computes itself a new state.
The inequality~\eqref{eq:QBound} guarantees that the agents have enough space to perform the computations, while~\eqref{eq:UBound} guarantees the time.
There are also $6Q$ extra steps coming from the retrieval phase, the marking of doomed cells, the fact that the $2Q$'th row from the top must consist of live cells, and some extra steps needed in the proofs of Section~\ref{sec:UnivProp} and Section~\ref{sec:SparseAmp}.
If we let $\xi \in \tra^+_\CA$ be the trajectory whose bottom row is $x$, this shows that $\pi(\xi)^0 = \eta^0$, that is, the bottom rows of $\pi(\xi)$ and $\eta$ coincide.

Now, we show by induction that $\pi(\xi)^t = \eta^t$ holds for all $t \in \N$.
More explicitly, we show that for all $i \in \Z$, the segment $\xi^{(t+1)U-1}_{[iQ,(i+1)Q-1]} \in \alp^Q$ is either a colony with age $U-1$, containing the state $\eta^t_i$ in its $\simu$ fields and $\eta^{t+1}_i$ in its $\out$ fields, or the blank segment $\B_\alp^Q$, which cannot border a colony in a non-blank state.
In the first case, if $\eta^{t+1}_i = B_\varalp$, then the cells of the colony are doomed, and it will be destroyed in the next step.
The latter case can only happen when $\eta^t_i = B_\varalp$.
If a doomed colony or a blank segment lies next to a colony that is not doomed, a new colony will be created in its place in the next $Q$ steps that form the retrieval phase.
During this phase, the colonies will also retrieve the state data of their neighbors.
During the computation phase, the colonies calculate new states for themselves, and the blocking cells keep the blank areas from spreading over the colonies.
From the induction hypothesis, and the fact that $B_\varalp$ is quiescent for $\varCA$, it follows that these states form the row $\eta^{t+1}$.
\end{proof}

If the bounds in Equations~\eqref{eq:QBound} and~\eqref{eq:UBound} hold, we denote $\Psi = \univ^{Q,U}_p$ (or just $\univ_p$, if $Q$ and $U$ are clear from the context or irrelevant), and call $\Psi$ a \emph{universal simulator automaton (with input $p$)}.
The choice for the $2Q$'th row from the top, as opposed to the top row, may seem arbitrary, but it is actually necessary in some of the results of Section~\ref{sec:UnivProp}.
In all applications of universal simulators, the dimensions $Q$ and $U$ are assumed to be large (the bound $Q \geq 10$ should be sufficient).
Since $P_\intp^T$ grows faster than $P_\intp^S$, the parameter $U$ will usually be orders of magnitude larger than $Q$.

In Figure~\ref{fig:Simulation}, we have illustrated the simulation of the three-neighbor binary XOR automaton by a universal simulator, in a schematic form.

\input{SimulationPic}

\subsection{Properties of Universal Simulation}
\label{sec:UnivProp}

Next, we show that simulation by a universal simulator automaton is very well-behaved.
For this section, let $\CA = \univ^{Q,U}_p$ be a universal simulator CA for the automaton $\intp^\infty_p$, using the simulation $(Q,U,(0,U-2Q),C,\pi)$.
Be begin with the following rather technical lemma that restricts the structure of those macro-cells that actually appear in nondeterministic trajectories of $\CA$.

\begin{lemma}
\label{lem:MacroCells}
Let $\eta \in \hat \tra^+_\CA$ be such that $P = \eta^{[U,2U-1]}_{[0,Q-1]} \in C$.
Then for all $t \in [Q,U-2Q]$, the row $P^t$ is a colony with age $t$ and simulated state $\pi(P) \in \varalp$.
Also, exactly one of the following conditions holds.
\begin{enumerate}
\item $\pi(P) \neq B_\varalp$, and for all $t \in [0,Q-1]$, $P^t$ is a colony with age $t$ and simulated state $\pi(P)$, and $P$ contains no co-creative cells.
Also, the $Q \times U$ rectangle below $P$ satisfies $\eta^{[0,U-1]}_{[0,Q-1]} \in C$.
\item $\pi(P) = B_\varalp$, and for all $t \in [0,Q-1]$ and $i \in [t-Q,t-1]$, the cell $\eta^{U + t}_i$ is consistent with $P$, and $\eta^{U + t}_{t-1}$ is right creative.
\item $\pi(P) = B_\varalp$, and for all $t \in [0,Q-1]$ and $i \in [Q-t,2Q-t-1]$, the cell $\eta^{U + t}_i$ is consistent with $P$, and $\eta^{U + t}_{Q-t}$ is left creative.
\item $\pi(P) = B_\varalp$, for all $t \in [0,Q-1]$ and $i \in [t-Q,t-1] \cup [Q-t,2Q-t-1]$, the cell $\eta^{U + t}_i$ is consistent with $P$.
Also, the cell $\eta^{U + t}_{t-1}$ ($\eta^{U + t}_{Q-t}$) is right (left, respectively) creative for all $t \in [0, \frac12Q]$.
\end{enumerate}
\end{lemma}

The lemma intuitively states that a macro-cell consists mostly of live cells that are consistent with it, and that there are four possibilities for the shape of its $Q$ lowest rows.
The four cases correspond to $P$ not being created, being created from the left, created from the right, and created from both directions, respectively.
Figure~\ref{fig:Simulation} contains examples of all the cases.

\begin{proof}
The first claim is an easy induction argument, since $P^{U-2Q}$ is of said form by the definition of $C$, and if $P^{t-1}$ is not of that form, then neither is $P^t$, by Observations~\ref{obs:Creative} and~\ref{obs:Inconsistent}.
Namely, if a cell $P^{t-1}_i$ is inconsistent with $P$, then $P^t_i$ must be co-creative, which is impossible since $t > Q$.

Consider then the border cells of $P$, and suppose that $P^t_0$ and $P^s_{Q-1}$ are left and right co-creative, respectively, and consistent with $P$, for some $t, s \in [1,Q]$.
Then necessarily $t = s = Q$ by Observation~\ref{obs:Creative}.
We can show by induction that $P^t_{t-1}$ is right co-creative and $P^t_{Q-t}$ is left co-creative for all $t \in [\frac12 Q, Q-1]$, but then the row $P^{\frac12 Q}$ has no preimage under $\hat \CA$, a contradiction.
Thus both borders cannot contain co-creative consistent cells.

Next, suppose that the right border cell $P^Q_{Q-1}$ is right co-creative and consistent with $P$.
Now we can inductively show, using the two observations and the rules of $\CA$, that for all $t \in [0,Q-1]$ and $i \in [t-Q,t-1]$, the cell $P^t_i$ is consistent with $P$ (this is an induction first on $i$, then on $t$).
Furthermore, Observation~\ref{obs:Creative} shows that $P^t_{t-1}$ is right creative and thus satisfies $\simu(P^{t+1}_t) = 0$.
This in turn implies that $\pi(P) = B_\varalp$, since $\simu$-bits do not change during the work period, and we are in the second case above.
Symmetrically, if the left cell $P^Q_0$ is left co-creative and consistent, we are in the third case.

Suppose finally that neither border cell is both consistent and co-creative.
Then the two observations show that $P^t_0$ and $P^t_{Q-1}$ are consistent for all $t \in [1,Q-1]$.
Furthermore, each row $P^t$ with $t \in [\frac12Q+1,Q-1]$ consists of cells consistent with $P$, and none of them are creative.

Observation~\ref{obs:Creative} gives two possibilities for the row $w = P^{\frac12Q}$: either $w_{\frac12Q-1}$ is right co-creative and $w_{\frac12Q}$ is left co-creative, in which case we can prove as above that the fourth case holds, or $w$ does not contain co-creative cells.
In the latter case, the induction can continue, and $P^t$ is a colony with age $t$ for all $t \in [0,\frac12Q]$.
It remains to show that in this case we have $\pi(P) \neq B_\alp$ and $R = \eta^{[0,U-1]}_{[0,Q-1]} \in C$, and for that, note that the top row of $R$ is a colony with age $U-1$ that cannot be doomed.
As in the proof of the first claim of this lemma, we can show that $R^t$ is a colony with age $t$ for all $t \in [Q,U-1]$, and thus $R \in C$ and the agent of $R$ correctly computes $\pi(P)$ from the different fields of $R$.
Now, if we had $\pi(P) = B_\alp$, then the agent of $R$ would have marked the cells of $R^{U-1}$ with the doom flag, a contradiction.
\end{proof}

As an immediate corollary, we obtain that two distinct macro-cells cannot overlap very much.
This corollary is later used to prove that in a trajectory our final sparse automaton, there cannot exist too many simultaneous simulation structures that are not consistent with each other.

\begin{corollary}
\label{cor:NoOverlap}
Let $\eta \in \hat \tra_\CA$, and suppose we have $\eta^{[t,t+U-1]}_{[i,i+Q-1]} \in C$ and $\eta^{[s,s+U-1]}_{[j,j+Q-1]} \in C$ with $|i-j| < Q$ and $|t-s| < U-3Q$.
Then $(i,t) = (j,s)$.
\end{corollary}

Using Lemma~\ref{lem:MacroCells} and the two observations from Section~\ref{sec:UnivConstr}, we can prove the strong rigidity and strong connectivity of the universal simulation.

\begin{lemma}
\label{lem:UnivRigidity}
The simulation of $\varCA = \intp^\infty_p$ by $\CA = \univ^{Q,U}_p$ is strongly rigid.
\end{lemma}

\begin{proof}
Let $\eta \in \hat \tra_\CA$ be a trajectory such that $P = \eta^{[U,2U-1]}_{[0,Q-1]} \in C$.
It is enough to prove that $\intp_p(\pi(\eta)^0_{-1},\pi(\eta)^0_0,\pi(\eta)^0_1) = \pi(P)$ holds if the simulated state of $P$ is not blank, or $\pi(P) \neq B_\varalp$.
By the first case of Lemma~\ref{lem:MacroCells}, we have that $R = \eta^{[0,U-1]}_{[0,Q-1]} \in C$ and that the agent of $R$ correctly computes $\pi(P)$ from the inputs it receives.
Thus it suffices to prove that the inputs are correct, that is,
\begin{equation}
\label{eq:NiceInputs}
\lmail(R^Q_i) = (\pi(\eta)^0_{-1})_i, \  \simu(R^Q_i) = (\pi(\eta)^0_0)_i, \  \rmail(R^Q_i) = (\pi(\eta)^0_1)_i
\end{equation}
for all $i \in [0, Q-1]$.
Since $\pi(\eta)^0_0 = \pi(R)$, the second condition holds for all $i$, and since the remaining cases are symmetric, it suffices to prove the first one.

Denote by $T = \eta^{[0, U-1]}_{[-Q,-1]}$ the $Q \times U$ rectangular pattern to the left of $R$.
We split the proof into three cases depending on the simulated state of $T$ and the structure of $R$.
See Figure~\ref{fig:UnivRigidity} for illustrations of the cases, where the dotted lines represent the cells $R^Q_i$ for $i \in [0, Q-1]$.

\paragraph{Case 1: $\pi(T) \neq B_\varalp$}
In this case, we necessarily have $T \in C$, and the first case of Lemma~\ref{lem:MacroCells} holds for $T$.
In particular, the bottom row of $T$ contains the correct simulated state on its $\simu$-bits and thus also on its $\lmail$ fields, or $\lmail(T^0_i) = (\pi(\eta)^0_{-1})_i$ for all $i \in [0, Q-1]$.
Also, the cell $c = T^Q_{Q-1}$ on the right border of $T$ (the small square in Figure~\ref{fig:UnivRigidity}) is not co-creative.
This implies that $R^Q_0$, the right neighbor of $c$, is not co-creative either, as otherwise the procedure \textsc{Valid} would return $\mtt{False}$ for $c$ (on line \algref{alg:Validity}{ln:BorderInvalid}), and the cell $T^{Q+1}_{Q-1}$ above $c$ would be blank, contradicting Lemma~\ref{lem:MacroCells}.
Thus $R$ is not created from the right, and Lemma~\ref{lem:MacroCells} implies that the horizontal segment $\eta^t_{[t-Q, t-1]}$ (between the dashed lines in Figure~\ref{fig:UnivRigidity}) consists of cells which are consistent with $T$ and $R$ for all $t \in [0, Q]$.
By induction on $t$, we can prove that $\lmail(T^t_{Q - t + i}) = (\pi(\eta)^0_{-1})_i$ for all $i \in [0, Q-1]$ and $t \in [0, Q]$.
Namely, each cell $\eta^t_i$ simply copies its $\lmail$ field from $\eta^{t-1}_{i-1}$.
We have $\lmail(R^Q_i) = (\pi(\eta)^0_{-1})_i$ in the final case $t = Q$, which is exactly the first claim of \eqref{eq:NiceInputs}.

\paragraph{Case 2: $\pi(T) = B_\varalp$ and $R$ is created from the right}
In other words, $R$ satisfies case 3 of Lemma~\ref{lem:MacroCells}.
This implies that for all $t \in [1, Q]$, the cell $R^t_{Q-t}$ is left creative (represented by the solid diagonal line in Figure~\ref{fig:UnivRigidity}) and $R^{t-1}_{Q-t}$ is blank.
We prove by induction on $t+i$ that for all $t \in [1, Q]$ and $i \in [Q-t, Q-1]$, we have $\lmail(R^t_i) = 0$. 
Namely, this holds for the left creative cells and their right neighbors (the base cases $i = Q-t$ and $i = Q-t+1$), which copy their $\lmail$ fields from blank cells (or cells that the local rule of $\hat \CA$ regards as blank).
The inductive step works as in Case 1.
When $t = Q$, we again have the first claim of \eqref{eq:NiceInputs}.

\paragraph{Case 3: $\pi(T) = B_\varalp$ and $R$ is not created from the right}
In this case, the left border cell $R^Q_0$ is not left blocking.
By induction, none of the cells $R^t_0$ for $t \in [Q, U-1]$ is left blocking.
We prove by induction on $i$ that all of the cells $T^{U-Q-t}_i$ for $t \in [1, Q]$ and $i \in [Q-t, Q-1]$ (the upper dashed triangle in Figure~\ref{fig:UnivRigidity}) are consistent with $T$.
Consider first a cell $T^{U-Q-t}_{Q-1} = \eta^{U-Q-t}_{-1}$ on the right border of $T$, which is the left neighbor of $R^{U-Q-t}_0$ on the left border of $R$.
These cells are represented by the small squares in Figure~\ref{fig:UnivRigidity}.
Since the cell $R^{U-Q-t+1}_0$ is consistent with $T$ but not left blocking, Observation~\ref{obs:Inconsistent} applied to $\delta = 1$ implies that $T^{U-Q-t}_{Q-1}$ is also consistent with $T$.
The inductive step is a similar application of Observation~\ref{obs:Inconsistent}.
In the final case $t = Q$, we see that every cell in the row $T^{U-2Q}$ (the horizontal dashed line) is consistent with $T$, so that $T \in C$ by the definition of $C$.
Now, $T$ cannot satisfy the first case of Lemma~\ref{lem:MacroCells} since $\pi(T) = B_\varalp$, and it cannot be created from the left for the same reasons as $R$ could not be created from the right in Case 1.
This in turn implies that the cell $T^t_{Q-t} = \eta^t_{-t}$ is left co-creative for all $t \in [1, \frac12Q]$ (the solid diagonal line in the figure).
Then we can prove similarly to Case 2 that $\lmail(\eta^t_i) = 0$ for all $t \in [1, Q]$ and $i \in [\max(-t,-Q+t), t-1]$ (the area between the lower dashed lines), which again implies the first claim of \eqref{eq:NiceInputs} when $t = Q$.
\end{proof}

\begin{figure}
\begin{center}
\begin{tikzpicture}[scale=.5]

  \node [left] () at (0,0) {$0$};
  \node [left] () at (0,3) {$Q$};
  \node [left] () at (0,6) {$U-2Q$};
  \node [left] () at (0,10) {$U$};

  \fill[black!20] (0,0) rectangle (3,6);
  \fill[black!20] (3,0) -- (3,10) -- (6,10) -- (6,3) -- cycle;

  \draw (2.8,2.9) rectangle (3,3.1);
  \node (c) at (1.5,3.5) {$c$};
  \draw[->] (c) -- (2.7,3.1);

  \draw[dashed] (0,0) -- (3,3);
  \draw[dashed] (3,0) -- (6,3);
  \draw[thick,dashed,->] (2,.5) -- (4,2.5) node [midway,above,sloped] () {$\lmail$};

  \draw (0,0) rectangle (3,10);
  \draw (3,0) rectangle (6,10);
  \draw[dotted] (3,3) -- (6,3);
  \node [above] () at (3,10) {Case 1};
  \node () at (1.5,5) {$T$};
  \node () at (4.5,5) {$R$};

\begin{scope}[xshift=7cm]

  \fill[black!20] (3,3) -- (3,10) -- (6,10) -- (6,0) -- cycle;
  \draw (3,3) -- (6,0);
  \draw[thick,dashed,->] (4.5,1.5) -- (5.5,2.5) node [midway,above,sloped] () {$0$};

  \draw (0,0) rectangle (3,10);
  \draw (3,0) rectangle (6,10);
  \draw[dotted] (3,3) -- (6,3);
  \node [above] () at (3,10) {Case 2};
  \node () at (1.5,5) {$T$};
  \node () at (4.5,5) {$R$};

\end{scope}

\begin{scope}[xshift=14cm]

  \fill[black!20] (0,3) -- (0,6) -- (3,9) -- (3,0) -- cycle;
  \fill[black!20] (3,0) -- (3,10) -- (6,10) -- (6,3) -- cycle;

  \draw[dashed] (3,9) -- (0,6) -- (3,6);
  \node (d) at (1.5,9) {$T^{U-Q-t}_{Q-1}$};
  \node (e) at (4.5,9) {$R^{U-Q-t}_0$};
  \draw (2.8,7.4) rectangle (3.2,7.6);
  \draw[->] (d) |- (2.7,7.5);
  \draw[->] (e) |- (3.3,7.5);

  \draw (3,0) -- (1.5,1.5);
  \draw[dashed] (1.5,1.5) -- (3,3);
  \draw[dashed] (3,0) -- (6,3);
  \draw[thick,dashed,->] (2.25,.75) -- (4,2.5) node [midway,above,sloped] () {$0$};

  \draw (0,0) rectangle (3,10);
  \draw (3,0) rectangle (6,10);
  \draw[dotted] (3,3) -- (6,3);
  \node [above] () at (3,10) {Case 3};
  \node () at (1.5,5) {$T$};
  \node () at (4.5,5) {$R$};

\end{scope}

\end{tikzpicture}
\end{center}
\caption{An illustration of the proof of Lemma~\ref{lem:UnivRigidity}, not drawn to scale. The gray areas represent cells that we have proved to be live.}
\label{fig:UnivRigidity}
\end{figure}

\begin{lemma}
\label{lem:UnivConnect}
The simulation of $\varCA = \intp^\infty_p$ by $\CA = \univ^{Q,U}_p$ is strongly connecting.
\end{lemma}

\begin{proof}
Let $\eta \in \hat \tra_\CA$ be such that $P = \eta^{[U,2U-1]}_{[0,Q-1]} \in C$ and $\pi(P)$ is $\delta$-demanding for some $\delta \in \{-1,0,1\}$.
In particular, we have $\pi(P) \neq B_\alp$, and Lemma~\ref{lem:MacroCells} then implies that $R = \eta^{[0,U-1]}_{[0,Q-1]} \in C$.
Then every cell of the rectangle $\eta^{[Q,2U-2Q]}_{[0,Q-1]}$ is consistent with $P$.
Also, the cells of the vertical column $V = \eta^{[U-2Q,2U-2Q]}_0$ on the left border of $P$ and $R$ are all $0$-demanding, so we have a directed path from the base of $P$ to that of $R$, which proves the claim when $\delta = 0$.

Suppose then that $\delta = -1$.
Since $\CA$ is strongly rigid by Lemma~\ref{lem:UnivRigidity}, we have $T = \eta^{[0,U-1]}_{[0,Q-1]} \in C$ and $\pi(T) \neq B_\alp$.
In particular, the base cell $T^{U-2Q}_0$ is consistent with $T$.
As in the proof of Lemma~\ref{lem:UnivRigidity}, we can also prove that the left border of $R$ does not contain left blocking cells.
Using Observation~\ref{obs:Inconsistent}, we can again prove by induction on $t$ that every cell $T^{U-Q-t}_{Q-t}$ for $t \in [1, Q]$ is consistent with $T$.
Furthermore, the cells are $(-1)$-demanding, and the cell $T^{U-2Q}_0$ of the case $t = Q$ is the base of $T$.
Combined with the aforementioned column $V$, we have obtained a directed path from the base of $P$ to that of $T$, which proves the claim when $\delta = -1$.
The case $\delta = 1$ is even simpler.
\end{proof}

Finally, we prove the analogue of Lemma~\ref{lem:SparseParents} from Section~\ref{sec:SparseProp} for the universal simulation.

\begin{lemma}
\label{lem:UnivParents}
Let $\eta \in \hat \tra_\CA$, and let $c = \eta^t_i$ be a live cell for some $(i,t) \in \Z^2$.
Then for some $s \in [t-2U,t-U+1]$ and $j \in [i-3Q+1,i+2Q-1]$, the rectangle $P = \eta^{[s,s+U-1]}_{[j,j+Q-1]}$ is a macro-cell in a non-blank state, and $c$ is directly connected to the base of $P$, and thus consistent with it.
\end{lemma}

\begin{proof}
Suppose first that $\age(c) = U-1$.
In this case, we claim that the result holds for $s = t - U + 1$ and some $j = j_\delta = i - \addr(c) + \delta Q$ for $\delta \in \{-1,0,1\}$.
Namely, we can prove by induction that for all $r \in [0,2Q-1]$, the number of cells on the row $\eta^{t-r}_{[j_0,j_0+Q-1]}$ that are directly connected to (and thus consistent with) $c$ is at least $\min(Q,r+1)$.
For $r = 2Q - 1$ this implies that the row $\eta^{t-2Q+1}_{[j_0,j_0+Q-1]}$ is a colony with age $U-2Q$, so that $R = \eta^{[t-U+1,t]}_{[j_0,j_0+Q-1]} \in C$.
If $\pi(R) \neq B_\alp$, we can choose $P = R$ and the claim holds.

Otherwise, $R$ is created from some direction, and we assume that it is created from the left, the other case being symmetric.
Since the left border of $R$ does not contain blocking cells, we can prove that for all $r \in [0,Q-1]$,
the row $\eta^{t-Q-r}_{[j_0-r,j_0+Q-1-r]}$ consists of cells that are directly connected to $c$.
In the final case $r = Q-1$, we see that $\eta^{t-2Q+1}_{[j_{-1},j_{-1}+Q-1]}$ is also a colony with age $U-2Q$.
Thus $T = \eta^{[t-U+1,t]}_{[j_{-1},j_{-1}+Q-1]} \in C$, and it is easily seen that $T$ cannot be created from either direction, implying that $\pi(T) \neq B_\alp$.
Then we can choose $P = T$.

Next, suppose that $a = \age(\eta^t_i) \neq U-1$.
If $a \in \{0\} \cup [Q+1,U-2]$, then Observation~\ref{obs:Inconsistent} implies that $\eta^{t-1}_i$ is consistent with $\eta^t_i$, and otherwise, Observation~\ref{obs:Creative} implies that one of $\eta^{t-1}_\delta$ for $\delta \in \{-1,0,1\}$ is.
Since the age of said consistent cell is $(a - 1) \bmod U$, a simple induction finishes the proof.
\end{proof}

\section{Amplification}
\label{sec:Amplification}

The main reason for constructing the universal simulator automaton is the following powerful result, which shows the existence of a wide class of amplifiers.

\begin{theorem}
\label{thm:Amplification}
Let $G$ be a polynomial CA transformation.
Then there exist computable sequences $(Q_n, U_n)_{n \in \N}$ in $\N^2$ and $(p_n)_{n \in \N}$ in $\{0, 1\}^*$ such that each $p_n$ is a program for the universal simulator $\univ^{Q_n, U_n}_{G(n+1, p_{n+1})}$.
\end{theorem}

In particular, if the transformation $G$ is such that $\intp^\infty_{G(n,p)}$ always simulates $\intp^\infty_p$, then the sequence $(\intp^\infty_{G(n, p_n)})_{n \in \N}$ is an amplifier.
We also note that the sequence $(Q_n, U_n)_{n \in \N}$ of dimensions could be chosen quite freely (here they are some polynomial functions of $n$), but we have no need for such fine-tuning in this article.

\begin{proof}
Let $P_G$ be a polynomial in two variables such that $|G(n, p)| \leq P_G(n, |p|)$ and $\size{G(n, p)} \leq P_G(n, \size{p})$ hold for all valid programs $p \in \{0,1\}^*$ and $n \in \N$.
Since the programming language defined in Section~\ref{sec:Interpreter} may be assumed to contain any particular polynomially computable function, we assume it contains the transformation $G$.

\begin{algorithm}[htp]
\caption{The modified parameters of $p_n$.} \label{alg:AmplifierParams}
\begin{algorithmic}[0]

\State $\mbf{num\    param}\ \lev = n$
\State $\mbf{num\    param}\ \mtt{CSize} = c_1 \cdot (\lev+1)^{c_2} - 1$
\State $\mbf{num\    param}\ \mtt{WPeriod} = c_3 \cdot (\lev+1)^{c_4} - 1$
\State $\mbf{string\ param}\ \mtt{SimProg} = G(\lev + 1, \textsc{IncLevel}(\mtt{This}))$

\end{algorithmic}
\end{algorithm}

We construct the programs $p_n$ by slightly modifying the program of the universal simulator CA.
In fact, it is enough to replace the parameter definitions of Algorithm~\ref{alg:ParamsFields} with Algorithm~\ref{alg:AmplifierParams}, and define one new procedure \textsc{IncLevel}.
The modified program contains the new parameter $\lev$, whose value is simply $n$, and the definitions of the other parameters are changed slightly.
In particular, the simulated program is now calculated from $p_n$ itself, using the keyword $\mtt{This}$ (which refers to $p_n$ as a bitstring) and the procedure \textsc{IncLevel}, which returns $p_{n+1}$ by simply incrementing the $\mtt{Level}$ parameter of its argument program by one.
The numbers $c_1, c_2, c_3, c_4 \in \N$ are some constants that only depend on $G$.
Of course, we define $Q_n = c_1 (n+1)^{c_2}$ and $U_n = c_3 (n+1)^{c_4}$ for all $n \in \N$.

We now claim that for some values of the constants $c_i$, the inequalities~\eqref{eq:QBound} and~\eqref{eq:UBound} hold with $Q$, $U$ and $p$ replaced by $Q_n$, $U_n$ and $G(n+1,p_{n+1})$ for all $n \in \N$.
For that, consider the program $p_{n+1}$.
Since the numbers are encoded in binary, its length is asymptotically $|p_{n+1}| = \log_2 n + \sum_{i=1}^4 \log_2(c_i) + \Oh(1)$, since the fields, procedure definitions and main body of the program are all independent of $n$ and the constants $c_i$.
Also, for the length of the states of $p_{n+1}$ we have $\size{p_{n+1}} = \log_2(c_1 c_3) + (c_2 + c_4) \log_2 n + \Oh(1)$.

Denote $q_n = G(n+1, p_{n+1})$.
We can now estimate the values $|q_n|$ and $\size{q_n}$ of the simulated automaton by
\[ |q_n| = P_G(n, \log_2 n + \sum_{i=1}^4 \log_2(c_i)) + \Oh(1) \]
and
\[ \size{q_n} = P_G(n, \log_2(c_1 c_3) + (c_2 + c_4) \log_2 n) + \Oh(1). \]
This implies that the complexities $P^S_\intp(|q_n|, \size{q_n})$ and $P^T_\intp(|q_n|, \size{q_n})$ are polynomial in $\log_2 n$ and the constants $c_i$.
It is clear that some values for the $c_i$ then exist that guarantee the inequalities $c_1 (n+1)^{c_2} \geq P^S_\intp(|q_n|, \size{q_n})$ and $c_3 (n+1)^{c_4} \geq 6 c_1 (n+1)^{c_2} + P^T_\intp(|q_n|, \size{q_n})$ for all $n \in \N$.
But these are exactly the bounds in~\eqref{eq:QBound} and~\eqref{eq:UBound}, and Lemma~\ref{lem:UnivSimu} finishes the proof.
\end{proof}

Theorem~\ref{thm:Amplification} reveals the motivation behind the definition of polynomial CA transformations.
Namely, a consistent transformation $G$ should be seen as a function that takes a cellular automaton $\CA$, together with an integer parameter $n \in \N$, and endows it with a property $P_n$, so that $P_n(G(n,\CA))$ holds.
Now, the theorem implies that if $G(n,\CA)$ is sufficiently similar to $\CA$, in the sense of being able to simulate it, there exists a single CA (the bottom of the amplifier) that, at least in some sense, satisfies $P_n$ for all $n \in \N$.
We could say that this single CA \emph{amplifies} the effects of $G$.

\section{The Sparse Amplifier}
\label{sec:SparseAmp}

\subsection{Properties of the Sparse Amplifier}

Let $N \subset \N$ be decidable in polynomial time.
Theorem~\ref{thm:Amplification} implies the existence of an amplifier $(G_s^N(n,\CA_n))_{n \in \N}$, where each $\CA_n$ is a universal simulator CA for $G_s^N(n+1,\CA_{n+1})$.
We call this the \emph{sparse amplifier} with effective set $N$, and proceed to study its properties.

We establish some notation for the rest of this section.
For convenience, we denote $\varCA_n = G_s^N(n,\CA_n)$, and let $\alp_n$ and $\varalp_n$ be the alphabets of $\CA_n$ and $\varCA_n$, respectively.
Let also $Q_n, U_n \in \N$ be given by Theorem~\ref{thm:Amplification}, and let $S^u_n = (Q_n,U_n,C^u_n,\pi^u_n)$ be the simulation of $\varCA_{n+1}$ by $\CA_n$.
Denote $M_n = 2n+1$ if $n \in N$, and $M_n = 1$ otherwise, and let $S^s_n = (M_n,M_n,C^s_n,\pi^s_n)$ be the simulation of $\CA_n$ by $\varCA_n$.
Denote $Q'_n = Q_n M_n$ and $U'_n = U_n M_n$, and let $S'_n = (Q'_n,U'_n,C'_n,\pi'_n)$ be the simulation of $\varCA_{n+1}$ by $\varCA_n$, that is, $S'_n = S^u_n \circ S^s_n$.
We also denote $B_n = \prod_{i=0}^{n-1} Q'_n$ and $W_n = \prod_{i=0}^{n-1} U'_n$, and let $S_n = (B_n,W_n,C_n,\pi_n)$ be the composed simulation of $\varCA_n$ by $\varCA_0$.
In particular, $S_0$ is the identity simulation of $\varCA_0$ by itself.
Thus, the $S^s_n$ are the sparse simulations, the $S^u_n$ the universal simulations, the $S'_n$ their compositions, and the $S_n$ the $n$-fold compositions of the $S'_n$.
The bottom of the amplifier, $\varCA_0$, is our main object of interest, and its unique ergodicity will be proved in the next section.
The simulation chain is drawn in Figure~\ref{fig:SimuChain}.

\begin{figure}
\begin{center}
\begin{tikzpicture}[node distance=2cm]

\node (0v) {$\varCA_0$};
\node (0) [below of=0v] {$\CA_0$};
\node (1v) [right of=0v] {$\varCA_1$};
\node (1) [below of=1v] {$\CA_1$};

\node (d) [right of=1v] {$\cdots$};

\node (nv) [right of=d] {$\varCA_{n-1}$};
\node (n) [below of=nv] {$\CA_{n-1}$};
\node (n1v) [right of=nv] {$\varCA_n$};
\node (n1) [below of=n1v] {$\CA_n$};

\draw[->] (0v) to [loop left] node {$S_0$} (0v);

\draw[->] (0v) to node [left] {$S^s_0$} (0);
\draw[->] (0) [bend right] to node [left,pos=.6] {$S^u_0$} (1v);
\draw[->] (0v) to node [below] {$S'_0 = S_1$} (1v);
\draw[->] (1v) to node [right] {$S^s_1$} (1);

\draw[->] (nv) to node [left] {$S^s_{n-1}$} (n);
\draw[->] (n) [bend right] to node [left,pos=.6] {$S^u_{n-1}$} (n1v);
\draw[->] (nv) to node [below] {$S'_{n-1}$} (n1v);
\draw[->] (n1v) to node [right] {$S^s_n$} (n1);

\draw[->] (0v) to [bend left] node [below] {$S_{n-1}$} (nv);
\draw[->] (0v) to [bend left] node [above] {$S_n$} (n1v);

\end{tikzpicture}
\end{center}
\caption{The chain of simulations induced by the sparse amplifier.}
\label{fig:SimuChain}
\end{figure}

We now proceed to prove the main properties of the sparse amplifier.
They mostly follow from the results of Section~\ref{sec:SparseProp} and Section~\ref{sec:UnivProp}.
In their proofs, we always assume that $N = \N$ unless otherwise noted, since the case $n \notin N$ is usually simpler than its converse and can be omitted.

\begin{lemma}
\label{lem:ChainRigid}
Every simulation $S'_n$ and $S_n$ is rigid and connecting.
\end{lemma}

\begin{proof}
This follows by induction from the rigidity (Lemma~\ref{lem:SparseRigidity} and Lemma~\ref{lem:UnivRigidity}) and connectivity (Lemma~\ref{lem:SparseConnect} and Lemma~\ref{lem:UnivConnect}) of the sparse and universal simulations, and Lemma~\ref{lem:CompProps}.
\end{proof}

The following result is analogous to Corollary~\ref{cor:NoOverlap} from Section~\ref{sec:UnivProp}.
It states that two distinct macro-cells of the simulation $S_n$ cannot overlap by more than $\frac13$ of their width and height, even in a nondeterministic trajectory of $\tilde \varCA_0$.
The bound $\frac13$ is not optimal, but suffices for our needs.
Since this proof quickly becomes quite technical, we split it into several smaller claims.
The reader should consult Figure~\ref{fig:OverlapLemma} for a visualization of the proof.

\input{OverlapLemma2}

\begin{lemma}
\label{lem:NoOverlap}
Let $\eta \in \tilde \tra_{\varCA_0}$ and $n \in \N$, and consider the simulation $S_n$ of $\varCA_n$ by $\varCA_0$.
Suppose that we have $P = \eta^{[0,W_n-1]}_{[0,B_n-1]} \in C_n$ and $R = \eta^{[t,t+W_n-1]}_{[i,i+B_n-1]} \in C_n$ for some $i \in [0,\frac23 B_n]$ and $t \in [0,\frac23 W_n]$.
Then $(i,t) = (0,0)$, so that $P = R$.
\end{lemma}

\begin{proof}
By induction on $n$.
For $n = 0$ the claim is trivial since the macro-cells of the simulation $S_0$ are just cells, so assume $n \geq 1$.
Let $\xi = \pi_{n-1}(\eta)$, $\zeta = \pi^s_{n-1}(\xi)$ and $\theta = \pi^s_{n-2}(\pi_{n-2}(\eta))$ (if $n \geq 2$).
Note that we have $\xi = \pi^u_{n-2}(\theta)$.
Since the simulations $S_k$ and $S^s_k$ are rigid and weakly rigid by Lemma~\ref{lem:ChainRigid} and Lemma~\ref{lem:SparseRigidity}, respectively, we have that $\xi \in \tilde \tra_{\varCA_{n-1}}$, $\zeta \in \hat \tra_{\CA_{n-1}}$ and $\theta \in \hat \tra_{\CA_{n-2}}$.
Next, we dissect the macro-cells $P$ and $R$ into smaller macro-cells of the lower-level simulation $S_{n-1}$.
For all $(j,s) \in [0,Q'_{n-1}-1] \times [0, U'_{n-1}-1]$, denote $P^{(j)}_{(s)} = P^{[sW_{n-1}, (s+1)W_{n-1}-1]}_{[jB_{n-1}, (j+1)B_{n-1}-1]}$, and similarly for $R$.

\begin{claim}
Let $(j,s) \in [0,Q_{n-1}-1] \times [Q_{n-1}, U_{n-1}-2Q_{n-1}]$.
Then the rectangles $P^{(s M_{n-1})}_{(j M_{n-1})}$ and $R^{(s M_{n-1})}_{(j M_{n-1})}$ are non-blank macro-cells of the simulation $S_{n-1}$.
\end{claim}

\begin{proof}[Proof of claim]
First, we have $T = \zeta^{[0,U_{n-1}-1]}_{[0,Q_{n-1}-1]} \in C^u_{n-1}$, since $P$ is mapped to $T$ by the state function $\pi^s_{n-1} \circ \pi_{n-1}$.
Lemma~\ref{lem:MacroCells} implies that the cell $\zeta^s_j$ is consistent with the macro-cell $T$.
But then the square pattern $T' = \xi^{[s M_{n-1}, (s+1) M_{n-1}-1]}_{[j M_{n-1}, (j+1) M_{n-1}-1]}$ is a macro-cell of the sparse simulation $S^s_{n-1}$, since it is mapped to $\zeta^s_j$ by the state function $\pi^s_{n-1}$.
This implies that $\xi^{s M_{n-1}}_{j M_{n-1}}$ is the base of the macro-cell $T'$.
We have $\xi^{s M_{n-1}}_{j M_{n-1}} = \pi_{n-1}(P^{(s M_{n-1})}_{(j M_{n-1})})$ by definition of $\xi$, implying that $P^{(s M_{n-1})}_{(j M_{n-1})}$ is a non-blank macro-cell of $S_{n-1}$.
The claim for $R^{(s M_{n-1})}_{(j M_{n-1})}$ is proved analogously, after translating $R$ to the origin.
\end{proof}

Consider now the trajectory $\zeta$ of $\hat \CA_{n-1}$.
From Lemma~\ref{lem:MacroCells} we in particular deduce that the vertical column $\zeta^{[Q_{n-1}+1,U_{n-1}-2Q_{n-1}]}_{\frac34 Q_{n-1}}$ consists of $0$-demanding cells.
Since the sparse simulation $\pi^s_{n-1}$ is connecting by Lemma~\ref{lem:SparseConnect}, this implies that in the trajectory $\xi$, there is a chain $(\xi^s_{j_s})_{s = Q'_{n-1}}^{U'_{n-1} - 2Q'_{n-1}}$ of non-blank cells such that $j_{m M_{n-1}} = \frac34 Q'_{n-1}$ for all $m \in [Q_{n-1}, U_{n-1} - 2Q_{n-1}]$, and $j_{s-1} = j_s + \delta$ for some $\delta \in \{-1,0,1\}$ for each $s$.
This chain corresponds, via the state function $\pi_{n-1}$, to a chain $(P^{(s)}_{(j_s)})_{s = Q'_{n-1}}^{U'_{n-1} - 2Q'_{n-1}}$ of non-blank macro-cells of $S_{n-1}$ in the trajectory $\eta$.
Similarly, there is a chain of non-blank macro-cells $(R^{(s)}_{(k_s)})_{s = Q'_{n-1}}^{U'_{n-1} - 2Q'_{n-1}}$ such that $k_{m M_{n-1}} = 0$ for all $m \in [2Q_{n-1}, U_{n-1} - 2Q_{n-1}]$, but $k_{Q'_{n-1}} = Q'_{n-1}-1$.
These chains are represented by dotted lines on the left hand side of Figure~\ref{fig:OverlapLemma}, and the macro-cells $P^{(s)}_{(s_j)}$ and $R^{(s)}_{(k_s)}$ are represented by the gray rectangles on the right hand side.

\begin{claim}
Let $s \in [Q'_{n-1}, U'_{n-1} - 2Q'_{n-1}]$.
If $n \geq 2$, then the rectangles $P^{(s-1)}_{(j_s)}$ and $R^{(s-1)}_{(k_s)}$ under the macro-cells of the aforementioned chains are themselves macro-cells of the simulation $S_{n-1}$.
\end{claim}

These lower rectangles are represented by the white rectangles on the right hand side of Figure~\ref{fig:OverlapLemma}.

\begin{proof}[Proof of claim]
First, the rectangle $T'' = \theta^{[s U_{n-2}, (s+1) U_{n-2}-1]}_{[j_s Q_{n-2}, (j_s+1) Q_{n-2}-1]}$ is a macro-cell of the universal simulation $S^u_{n-2}$ in a non-blank state, since it is mapped to the non-blank cell $\xi^s_{j_s}$ by the state function $\pi^u_{n-2}$.
By the first item of Lemma~\ref{lem:MacroCells}, the rectangle $\theta^{[(s-1) U_{n-2}, U_{n-2}-1]}_{[j_s Q_{n-2}, (j_s+1) Q_{n-2}-1]}$ under $T''$ is also a macro-cell of $S^u_{n-2}$.
But these cells are the $\pi^s_{n-2} \circ \pi_{n-2}$-images of the corresponding rectangles $P^{(s-1)}_{(j_s)}$, which are then macro-cells of $S_{n-1}$ (not necessarily in a non-blank state).
The case of the $R^{(s-1)}_{(k_s)}$ is again similar.
\end{proof}

Since the rectangle $R^{(3Q'_{n-1})}_{(k_{3Q'_{n-1}})}$ lies to the left of every $P^{(s)}_{(j_s)}$ but $R^{(2Q'_{n-1})}_{(k_{2Q'_{n-1}})}$ does not, there are some $s, r \in [Q'_{n-1},U'-2Q'_{n-1}]$ and $\delta, \epsilon \in \{0,1\}$ such that the rectangles $P^{(s-\delta)}_{(j_s)}$ and $R^{(r-\epsilon)}_{(k_r)}$ overlap by more than $\frac13 B_{n-1}$ cells horizontally and $\frac13 W_{n-1}$ cells vertically.
If $n \geq 2$, both rectangles are macro-cells of the simulation $S_{n-1}$, and the induction hypothesis implies that $P^{(s-\delta)}_{(j_s)} = R^{(r-\epsilon)}_{(k_r)}$.
Thus $i$ is divisible by $B_{n-1}$ and $t$ by $W_{n-1}$, and we denote $i' = i / B_{n-1}$ and $t' = t / W_{n-1}$.
If $n = 1$, we have $B_{n-1} = W_{n-1} = 1$, so the divisibility property holds trivially.

Now, the rectangle $R' = \xi^{[t', t' + U'_{n-1}]}_{[i', i' + Q'_{n-1}]}$ is mapped to $R$ by the state function $\pi_{n-1}$, and thus is a macro-cell of the simulation $S'_{n-1}$ of $\varCA_n$ by $\varCA_{n-1}$.
Similarly, we have $P' = \xi^{[0, 0 + U'_{n-1}]}_{[0, 0 + Q'_{n-1}]} \in C'_{n-1}$.
The two chains $(\xi^s_{j_s})_{s = Q'_{n-1}}^{U'_{n-1} - 2Q'_{n-1}}$ and $(\xi^{t' + s}_{i' + k_s})_{s = Q'_{n-1}}^{U'_{n-1} - 2Q'_{n-1}}$ cross each other in the trajectory $\xi$, so that the topmost cells $\xi^{U'_{n-1} - 2Q'_{n-1}}_0$ and $\xi^{t' + U'_{n-1} - 2Q'_{n-1}}_{i'}$ on the left borders of $P'$ and $R'$ are weakly connected.
Since the two cells are bases of macro-cells of the sparse simulation $\pi^s_{n-1}$, Lemma~\ref{lem:SparseEvenCells} implies that $i'$ and $t'$ are both divisible by $M_{n-1}$, and we denote $i'' = i' / M_{n-1}$ and $t'' = t' / M_{n-1}$.
Then the rectangles $R'' = \zeta^{[t'', t''+U_{n-1}-1]}_{[i'', i''+Q_{n-1}-1]}$ and $P'' = \zeta^{[0, U_{n-1}-1]}_{[0, Q_{n-1}-1]}$ are macro-cells of the universal simulation $S^u_{n-1}$.
Finally, we have $i'' \leq \frac23 Q_{n-1}$ and $t'' \leq \frac23 U_{n-1}$ by assumption, so the claim follows from Corollary~\ref{cor:NoOverlap}.
\end{proof}

Finally, we claim that a non-blank macro-cell of any level cannot be very far from a non-blank macro-cell of a higher level.
The result follows from the rigidity and connectivity of our simulations (Lemma~\ref{lem:ChainRigid}), together with Lemma~\ref{lem:SparseParents} and Lemma~\ref{lem:UnivParents}, which proved the analogous results for the sparse and universal simulations.

\begin{lemma}
\label{lem:Parents}
Let $n \in \N$, and let $\eta \in \tilde \tra_{\varCA_0}$.
Suppose that $P = \eta^{[t,t+W_n-1]}_{[i,i+B_n-1]}$ is a non-blank macro-cell of $S_n$ for some $(i,t) \in \Z^2$.
Then there exist coordinates $j \in [i-3B_{n+1},i+2B_{n+1}]$ and $s \in [t-2W_{n+1},t-W_{n+1}+1]$ such that $R = \eta^{[s,s+W_{n+1}-1]}_{[j,j+B_{n+1}-1]}$ is a non-blank macro-cell of $S_{n+1}$ and the base of $P$ is directly connected to that of $R$. Furthermore, $B_n$ divides $|i-j|$ and $W_n$ divides $|t-s|$.
\end{lemma}

By the above and a simple induction we obtain the following lemma, which states that non-blank cells are found only in the vicinity of correct macro-cells.

\begin{corollary}
\label{cor:Parents}
Let $n \in \N$, and let $\eta \in \tilde \tra_{\varCA_0}$.
Suppose that $\eta^i_t \neq B_{\varalp_0}$ for some $(i,t) \in \Z^2$.
Then there exist coordinates $j \in [i-4B_n,i+3B_n]$ and $s \in [t-3W_n,t]$ such that $P = \eta^{[s,s+W_n-1]}_{[j,j+B_n-1]}$ is a non-blank macro-cell of $S_n$.
\end{corollary}

\subsection{Unique Ergodicity}

We are now ready to prove our main result, the unique ergodicity of $\varCA_0$, in the case that the effective set $N$ is infinite.

\begin{theorem}
\label{thm:MainResult}
The cellular automaton $\varCA_0 = G_s^N(0,\CA_0)$ is not nilpotent.
It is uniquely ergodic if and only if $N$ is infinite.
\end{theorem}

\begin{proof}
First, we need to show that $\Phi_0$ is not nilpotent.
Let $n \in \N$, and let $\eta \in \tra^+_{\CA_n}$ be a one-directional trajectory of $\CA_n$ such that the central cell $\eta^0_0$ is not blank.
By the definition of the simulation $S_n$, there exists $\xi \in \tra^+_{\CA_0}$ such that $\pi_n(\xi) = \eta$.
This implies that $\xi^{[0, W_n-1]}_{[0, B_n-1]} \in C_n$, and by the definition of $C_n$, the base cell $\xi^{W_n - 2 B_n}_0$ is not blank.
Since $W_n - 2 B_n$ grows arbitrarily large with $n$, the CA $\CA_0$ is not nilpotent by definition.

Second, if there exists $m \in \N$ such that $n \notin N$ for all $n \geq m$, then there is a trajectory $\eta \in \tra^+_{\varCA_m}$ such that $\eta^t_i \neq B_{\varalp_n}$ for all $(i,t) \in \Z \times \N$.
Then for the trajectory $\xi \in \tra^+_{\varCA_0}$ with $\pi_n(\xi) = \eta$ we clearly have
\[ \limsup_{s \longrightarrow \infty} \frac{1}{s} \left| \left\{ t \in [0,s-1] \;\middle|\; \xi^t_0 \neq B_{\varalp_0} \right\} \right| \geq W_m^{-1} > 0. \]
Corollary~\ref{cor:UEChara} now shows that $\varCA_0$ is not uniquely ergodic.

Finally, suppose that $N$ is infinite, and let $\eta \in \tilde \tra_{\varCA_0}$ be a two-directional trajectory.
For $n \in \N$, we say that two macro-cells $\eta^{[t,t+W_n-1]}_{[i,i+B_n-1]}$ and $\eta^{[s,s+W_n-1]}_{[j,j+B_n-1]}$ of $S_n$ are \emph{spatio-temporally consistent} if $i \equiv j \bmod B_{n-1}$ and $t \equiv s \bmod W_{n-1}$.
We first use the pigeonhole principle to prove the following auxiliary result.

\begin{claim}
\label{claim:Pigeonhole}
There exists a number $q \in \N$ with the following property: Let $n \geq 1$, and let $K \subset [-4B_{n-1}, 3B_{n-1}] \times [-W_n, 2W_n - 1]$ be such that for all $(i,t) \in K$, the rectangle $P^{(t)}_{(i)} = \eta^{[t,t+W_{n-1}-1]}_{[i,i+B_{n-1}-1]}$ is a non-blank macro-cell of the simulation $S_{n-1}$, and for all distinct $(i,t), (j,s) \in K$, $P^{(t)}_{(i)}$ and $P^{(s)}_{(j)}$ are spatio-temporally inconsistent.
Then we have $|K| \leq q$.
\end{claim}

\begin{proof}[Proof of claim]
First, Lemma~\ref{lem:Parents} implies that for any rectangle $P^{(t)}_{(i)}$ for $(i,t) \in [-4B_{n-1}, 3B_{n-1} - 1] \times [-W_n, 2W_n - 1]$ which is a non-blank macro-cell of $S_{n-1}$, there exist $j \in [-4(Q'_n+1) B_{n-1},(3Q'_n+2) B_{n-1}]$ and $s \in [-3W_n, W_n]$ such that $R^{(s)}_{(j)} = \eta^{[s,s+W_n-1]}_{[j,j+B_n-1]}$ is a non-blank macro-cell of $S_n$, and $i \equiv j \bmod B_{n-1}$ and $t \equiv s \bmod W_{n-1}$ hold.
We associate one such pair $(j, s)$ to $(i, t)$, and denote $(j, s) = F(i, t)$.
Note that we in particular have $j \in [-5B_n, 4B_n]$ in the above.
Lemma~\ref{lem:NoOverlap}, together with a pigeonhole argument, then implies that there exists $q \in \N$ independent of $n$ such that at most $q$ of the rectangles $R^{(s)}_{(j)}$ for $(j,s) \in [-5B_n,4B_n] \times [-3W_n,W_n]$ are macro-cells of $S_{n+1}$.
Thus, the function $F : K \to [-5B_n, 4B_n] \times [-3W_n, W_n]$ has an image of size at most $q$.
Finally, if $F(i,t) = F(i',t')$ for some $(i,t), (i',t') \in K$, then the bases macro-cells $P^{(t)}_{(i)}$ and $P^{(t')}_{(i')}$ are spatio-temporally consistent, so that $(i,t) = (i',t')$ by assumption.
Thus $F$ is injective, which implies $|K| \geq q$.
\end{proof}

\begin{claim}
\label{claim:Density}
Let $n \in \N$ be such that $n+1 \in N$, let $\ell \in \N$, and denote
\[ L = \{ t \in [\ell W_{n+1}, (\ell+1) W_{n+1}-1] \;|\; \eta^t_0 \neq B_{\varalp_0} \}. \]
Then we have $|L| \leq \frac{48 q W_{n+1}}{n}$.
\end{claim}

\begin{proof}[Proof of claim]
Denote $A = [-4B_n,3B_n] \times [\ell W_{n+1} - 3W_n, (\ell+1) W_{n+1}-1]$, and let $(j,s) \in A$ be such that $P = \eta^{[s,s+W_n-1]}_{[j,j+B_n-1]}$ is a non-blank macro-cell of $S_n$.
Without loss of generality, assume that $j$ is divisible by $B_n$ and $s$ by $W_n$, and let $\xi = \pi_n(\eta)$.
By Lemma~\ref{lem:ChainRigid} we have $\xi \in \tilde \tra_{\varCA_n}$.
Then, denoting $j' = j B_n^{-1}$ and $s' = s W_n^{-1}$, we have $\xi^{s'}_{j'} \neq B_{\varalp_n}$.
By Lemma~\ref{lem:Sparsity}, the set $L'$ of those coordinates $(k',r') \in [-3,2] \times [\ell U'_{n+1}-3, (\ell+1) U'_{n+1}-1]$ for which $\xi^{r'}_{k'}$ is connected to $\xi^{s'}_{j'}$ has size at most $\frac{15}{n}(U'_{n+1}+3)$.
Now, if $\eta^{[k, k+W_n-1]}_{[r,r+B_n-1]}$ is a non-blank macro-cell of $S_n$ which is spatio-temporally consistent with $P$, then $k' = k B_n^{-1}$ and $r' = r W_n^{-1}$ are both integers, and hence $(k',r') \in L'$.
Thus the number of such coordinates $(k,r)$ is also at most $\frac{15}{n}(U'_{n+1}+3)$.

Denote by $H$ the set of those $(j,s) \in A$ for which $\eta^{[s,s+W_n-1]}_{[j,j+B_n-1]}$ is a non-blank macro-cell of $S_n$.
By Claim~\ref{claim:Pigeonhole}, there are at most $q$ classes of spatio-temporally consistent macro-cells in $A$, so the cardinality of $H$ is at most $\frac{15q}{n}(U'_{n+1}+3)$.

Let now $t \in L$.
By Corollary~\ref{cor:Parents}, there exists $(j,s) \in [-4B_n,3B_n] \times [t-3W_n,t]$ such that $\eta^{[s,s+W_n-1]}_{[j,j+B_n-1]}$ is a non-blank macro-cell of $S_n$, that is, $(j,s) \in H$.
Hence $L \subset \{ t \in [s, s+3W_n] \;|\; (j,s) \in H \}$, implying that
\[ |L| \leq 3W_n |H| \leq \frac{45 q W_n}{n}(U'_{n+1}+3) \leq \frac{48 q W_{n+1}}{n}, \]
since $W_n \leq W_{n+1}$.
\end{proof}

Since the effective set $N$ contains arbitrarily large numbers, the above claim directly implies that
\[ \limsup_{s \rightarrow \infty} \frac{1}{n} \left| \left\{ t \in [0,s-1] \;\middle|\; \eta^t_0 \neq B_{\varalp_0} \right\} \right| = 0, \]
and Corollary~\ref{cor:UEChara} then shows that $\varCA_0$ is uniquely ergodic.
\end{proof}

\section{Further Results}
\label{sec:Further}

In this section, we explore some extensions and variants of Theorem~\ref{thm:MainResult} that can be proved using an amplifier construction.
We also state some undecidability results regarding them.

\subsection{Asymptotic Nilpotency in Density}

The product topology, induced by the Cantor metric, is usually seen as the natural environment for cellular automata.
However, since it places a heavy emphasis on the central coordinates of a configuration, more `global' topologies have been defined, one of the most well-known being the Besicovitch topology.

\begin{definition}
The \emph{Besicovitch pseudometric} $d_B : \alp^\Z \times \alp^\Z \to \R$ is defined by
\[ d_B(x,y) = \limsup_{n \longrightarrow \infty} \frac{1}{2n+1} \left| \left\{ i \in [-n,n] \;\middle|\; x_i \neq y_i \right\} \right| \]
for $x, y \in \alp^\Z$.
The \emph{Besicovitch class} $[x]$ of a configuration $x \in \alp^\Z$ is the set $\{ y \in \alp^\Z \;|\; d_B(x,y) = 0 \}$.
\end{definition}

The function $d_B$ is a pseudometric on $\alp^\Z$, and thus induces a metric on the set $\{ [x] \;|\; x \in \alp^\Z \}$ of Besicovitch classes of configurations.
It is known that this metric space it complete, but not compact \cite{BlFoKu97}.
Instead of comparing the central coordinates, it measures the asymptotic density of the set of differing cells.
The pseudometric was first introduced in \cite{CaFoMaMa97}.

The following result appeared in \cite{To12}, and we repeat the proof here for completeness.

\begin{proposition}
\label{prop:BesNilpotency}
Let $\CA : \alp^\Z \to \alp^\Z$ be a cellular automaton such that for all $x \in \alp^\Z$ there exists $n \in \N$ with $\CA^n(x) \in [\INF B_\alp^\infty]$.
Then $\CA$ is nilpotent.
\end{proposition}

\begin{proof}
Let $x \in \alp^\Z$ be a generic point for the uniform Bernoulli distribution on $\alp$ (see \cite{DeGrSi76}).
This means that
\begin{equation}
\label{eq:Generic}
\lim_{n \rightarrow \infty} \frac{1}{2n+1} \left| \left\{ i \in [-n,n] \;\middle|\; x_{[i,i+|w|-1]} = w \right\} \right| = |\alp|^{-|w|}
\end{equation}
holds for all $w \in \alp^*$.
By assumption, $\CA^n(x) \in [\INF B_\alp^\infty]$ holds for some $n \in \N$.
Let $w \in \alp^{2n+1}$ be arbitrary.
If we had $\CA^n(w) \neq B_\alp$, then $d_B(\CA^n(x), \INF B_\alp^\infty) \geq |\alp|^{-2n-1}$ by~\eqref{eq:Generic}, a contradiction.
Thus we have $\CA^n(w) = B_\alp$, and since $w \in \alp^{2n+1}$ was arbitrary, we have $\CA^n(\alp^\Z) = \{ \INF B_\alp^\infty \}$.
\end{proof}

In particular, the above holds if $\CA^n(x)$ actually equals $\INF B_\alp^\infty$.
It is also well known (see for instance \cite{CuJaYu89}) that if for all $\epsilon > 0$ there exists $n \in \N$ such that $d_C(\CA^n(x),\INF B_\alp^\infty) < \epsilon$ for all $x \in \alp^\Z$, then $\Omega_\CA = \{ \INF B_\alp^\infty \}$, and thus $\CA$ is nilpotent by Proposition~\ref{prop:Nilpotent}.
Next, we show that the Besicovitch analogue of this result does not hold, and in fact the automaton $\varCA_0$ from the sparse amplifier is a counterexample.
First, we give a name for the property.

\begin{definition}
Let $\CA : \alp^\Z \to \alp^\Z$ be a cellular automaton.
Suppose that for all $\epsilon > 0$ there exists $n \in \N$ such that $d_B(\CA^n(x),\INF B_\alp^\infty) < \epsilon$ holds for all $x \in \alp^\Z$.
Then we say $\CA$ is \emph{asymptotically nilpotent in density} (AND for short).
\end{definition}

In other words, a cellular automaton is asymptotically nilpotent in density if the asymptotic density of non-blank cells in a configuration converges to $0$ under the action of the CA, and the speed of the convergence is uniform.
The condition of uniform convergence can be dropped, as shown by the following characterization.

\begin{lemma}
\label{lem:AND}
Let $\CA : \alp^\Z \to \alp^\Z$ be a cellular automaton.
Then $\CA$ is asymptotically nilpotent in density if and only if $\Omega_\CA \subset [\INF B_\alp^\infty]$, if and only if
\begin{equation}
\label{eq:LimitDensity}
\lim_{\ell \longrightarrow \infty} \max_{w \in \B_\ell(\Omega_\CA)} \frac{1}{\ell} |\{ i \in [0,\ell-1] \;|\; w_i \neq B_\alp \}| = 0.
\end{equation}
\end{lemma}

\begin{proof}
First, if $\CA$ is AND, it is clear that $\Omega_\CA \subset [\INF B_\alp^\infty]$ holds.
Suppose then that \eqref{eq:LimitDensity} holds, and let $\epsilon > 0$.
Then there exists a length $\ell \in \N$ such that $|\{ i \in [0,\ell-1] \;|\; w_i \neq B_\alp \}| < \epsilon \cdot \ell$ holds for all $w \in \B_\ell(\Omega_\CA)$.
It is known (see again \cite{CuJaYu89}) that there now exists $n \in \N$ such that $\B_\ell(\CA^n(\alp^\Z)) = \B_\ell(\Omega_\CA)$.
But this implies that $d_B(\CA^n(x), \INF B_\alp^\infty) < \epsilon$ for all $x \in \alp^\Z$, and since $\epsilon$ was arbitrary, $\CA$ is AND.

We still need to show that $\Omega_\CA \subset [\INF B_\alp^\infty]$ implies \eqref{eq:LimitDensity}, and for that, consider the shift map $\sigma : \Omega_\CA \to \Omega_\CA$, which is a cellular automaton on $\Omega_\CA$.
Using the notation of Section~\ref{sec:Prelim}, we have $d_\sigma(B_\alp,x) = d_{\sigma^{-1}}(B_\alp,x) = 1$ for all $x \in \Omega_\CA$, and Proposition~\ref{prop:UEChara} implies that the convergence of the limits is uniform in $x$.
But this is equivalent to \eqref{eq:LimitDensity}, and we are done.
\end{proof}

We only sketch the proof of the following result, since it is mostly analogous to that of Theorem~\ref{thm:MainResult}.

\begin{proposition}
\label{prop:DensityConvergence}
The automaton $\varCA_0$ is asymptotically nilpotent in density if and only if the effective set $N$ is infinite.
\end{proposition}

\begin{proof}[Proof sketch]
If $N$ is finite, then $\varCA_0$ is not AND, for the same reason as in Theorem~\ref{thm:MainResult}.

Suppose now that $N$ is infinite.
By Lemma~\ref{lem:AND}, it suffices to prove that $\Omega_{\varCA_0} \subset [\INF B_{\varalp_0}^\infty]$, so let $x \in \Omega_{\varCA_0}$ be arbitrary.
Then there exists a two-directional trajectory $\eta \in \tra_{\varCA_0}$ with $\eta^0 = x$.
It is easy to see that the variant of Lemma~\ref{lem:Sparsity}, where $\eta^s_0$ is replaced by $\eta^0_s$, holds.
We can also prove the variant of Claim~\ref{claim:Pigeonhole} in the proof of Theorem~\ref{thm:MainResult} where $K \subset [-B_n, B_n] \times [-3W_{n-1}, W_{n-1} - 1]$.
Finally, we can prove a variant of Claim~\ref{claim:Density} where
\[ L = \{ i \in [\ell B_{n+1}, (\ell+1) B_{n+1}-1] \;|\; \eta^0_i \neq B_{\varalp_0} \}, \]
and the bound is modified accordingly.
This shows that $x = \eta^0 \in [\INF B_{\varalp_0}^\infty]$, so that we have $\Omega_{\varCA_0} \subset [\INF B_{\varalp_0}^\infty]$.
\end{proof}

The main result of \cite{GuRi08} states that if a cellular automaton $\CA : \alp^\Z \to \alp^\Z$ satisfies $\CA^n(x) \stackrel{n \rightarrow \infty}{\longrightarrow} \INF B_\alp^\infty$ for all $x \in \alp^\Z$ with respect to the Cantor metric $d_C$, then it is nilpotent.
The above proposition in particular shows that the $d_B$-analogue of this result does not hold.
In the same article, it was asked whether the condition that the $d_C$-orbit closure $\overline{\{\CA^n(x) \;|\; n \in \N\}}$ or every $x \in \alp^\Z$ contains the uniform configuration $\INF B_\alp^\infty$ implies nilpotency, and Theorem~\ref{thm:MainResult} clearly disproves this.
Finally, \cite[Corollary 31]{GuRi10} states that if the limit set $\Omega_\CA$ contains only $B_\alp$-finite configurations, then $\CA$ is nilpotent.
Lemma~\ref{lem:AND} and Proposition~\ref{prop:DensityConvergence} show that this result is strict, in the sense that $\Omega_\CA$ containing only configurations where the asymptotic density of blank symbols is $1$ is not enough to guarantee nilpotency.

Next, we show that Proposition~\ref{prop:DensityConvergence} is `only' an artifact of our construction, since the notions of unique ergodicity and AND are independent.
For this, we make two modifications to the sparsification transformation $G_s$.
More explicitly, we define two new transformations $G_1$ and $G_2$ that add some new fields and rules to their input program, and replace $G_s$ with either $G_1 \circ G_s$ or $G_2 \circ G_s$.
When defining $G_1$ and $G_2$, we thus assume that the input program is already sparsified by $G_s$, and thus contains a numeric parameter $\mtt{N}$ and a numeric field $\cou$.
These transformations rename the fields, parameters and procedures of the input program if necessary, except for those defined by $G_s$, which they use during their own execution.

First, the application of $G_1$ to $n \in \N$ and a sparsified program $p$ is defined in Algorithm~\ref{alg:NotUE}.
The transformation $G_1$ adds one new counter, $\ccou$, ranging from $0$ to $2n$, that is initialized to $1$ at each base cell (a right-moving cell with $\cou = 0$).
The counter field is then incremented by one for $2n$ steps, after which it remains at $0$, until the cell becomes a base cell again.
The new counter operates on a different `layer' as the sparsified automaton, and does not affect its dynamics.
Geometrically, $G_1$ adds to all trajectories of $G_s$ vertical lines that connect each base cell to the one above it, as shown in Figure~\ref{fig:SparsifiedVarXOR}.

\begin{algorithm}[htp]
\caption{The transformed program $G_1(n, p)$.}\label{alg:NotUE}
\begin{algorithmic}[0]
\State $\mbf{num\    field}\ \ccou \leq 2 \, \mtt{N}$
\State $p$ \Comment{Definitions and body of $p$}
\If{$\kind = {\Rightarrow} \wedge \cou = 0$}
	\State $\ccou \gets 1$ \Comment{The base of a macro-cell stays non-blank\ldots}
\ElsIf{$\ccou > 0$}
	\State $\ccou \gets \ccou + 1 \bmod 2 \, \mtt{N}$ \Comment{\ldots{}for exactly $2 \, \mtt{N}$ steps}
\Else
	\State $\ccou \gets 0$
\EndIf

\end{algorithmic}
\end{algorithm}

In the second transformation, $G_2$, we wish to introduce horizontal lines instead of vertical ones.
This is a little more complex, and to define $G_2$, we need the following auxiliary concepts.

\begin{definition}
We inductively define a sequence $(L(n))_{n \geq 1}$ of finite sets $L(n) \subset \Z^2$ as follows.
First, let $L(1) = \{(0,0)\}$ and $L(2) = \{(0,0), (-1,1), (0,1), (1,1)\}$.
For $n \geq 3$, we denote $n^* = \left\lceil \frac{n+1}{2} \right\rceil$ and $n_* = \left\lfloor \frac{n+1}{2} \right\rfloor$, and define
\[ L(n) = \{ (k, \pm k) \;|\; 0 \leq k \leq n^* \} \cup L(n^*) + (n_*, -n_*) \cup L(n_*)  + (n^*, n^*). \]
For each $n \geq 1$, define also a configuration $\eta_n \in (L(n) \cup \{ \# \})^{\Z^2}$ by
\[ (\eta_n)^t_i = \left\{ \begin{array}{ll}
    (i, t), & \mbox{if~} (i, t) \in L(n), \\
    \#, & \mbox{otherwise,}
\end{array} \right. \]
and a function $f_n : (L(n) \cup \{\#\})^3 \to L(n) \cup \{\#\}$ by
\[ f_n(\vec u, \vec v, \vec w) = \left\{ \begin{array}{ll}
    (i, t) \in L(n) \setminus \{ \vec 0 \}, & \mbox{if~} (\vec u, \vec v, \vec w) = g_n(\eta_n)^{t-1}_{[i-1,i+1]}, \\
    \#, & \mbox{otherwise.}
\end{array} \right. \]
\end{definition}

The sets $L(n)$ and the functions $f_n$ satisfy the following properties, which are not hard to prove:
\begin{enumerate}
\item $[-n+1, n-1] \times \{n-1\} \subset L(n) \subset [-n+1, n-1] \times [0, n-1]$,
\item the cardinality of the intersection of $L(n)$ with any vertical column is logarithmic in $n$,
\item the sets $L(n)$ and functions $f_n$ are computable in polynomial time uniformly in $n$, and
\item the restriction $g_n(\eta_n)^{[0, \infty)}$ to the upper half-plane is a one-directional trajectory of the cellular automaton with local function $f_n$.
\end{enumerate}
See Figure~\ref{fig:LnPic} for a visualization of $L(n)$ for different values of $n$.

\input{LnPic}

The application of the second transformation $G_2$ to $n \in \N$ and $p$ is defined in Algorithm~\ref{alg:NotAND}.
Geometrically, the idea is to draw on every macro-cell of the level-$n$ sparse simulation a copy of the pattern $L(n)$, rooted at the base cell.
This is achieved by the new field $\celem$ with values in $L(\mtt{N}) \cup \{\#\}$, of which $\#$ is presented by zeros, and the functions $f_n$.
The field $\celem$ of a base cell gets the value $(0,0) \in L(\mtt{N})$, and on subsequent steps, the new value of $\celem$ is computed using $f_{\mtt{N}}$.
See Figure~\ref{fig:SparsifiedVarXOR} for a visualization.

\begin{algorithm}[htp]
\caption{The transformed program $G_2(n, p)$.}\label{alg:NotAND}
\begin{algorithmic}[1]

\State $\mbf{enum\ field}\ \celem \in L(\mtt{N}) \cup \{\#\}$
\State $p$ \Comment{Definitions and body of $p$}
\If{$\kind = {\Rightarrow} \wedge \cou = 0$}
	\State $\celem \gets (0,0)$ \Comment{From the base of a macro-cell\ldots}
\EndIf
\State $\celem \gets f_{\mtt{N}}(\celem^-, \celem, \celem^+)$ \Comment{\ldots{}construct the set $L(\mtt{N})$}

\end{algorithmic}
\end{algorithm}

Now, let $N \subset \N$ be decidable in polynomial time, and consider the transformations
\[ G_i^N(n, p) = \left\{ \begin{array}{cl}
	G_s(n, p), & \mbox{if~} n \in N, \\
	(G_i \circ G_s)(n, p), & \mbox{if~} n \notin N
\end{array} \right. \]
for $i \in \{1,2\}$.
Let the cellular automata $\vvarCA_i : \vvaralp_i^\Z \to \vvaralp_i^\Z$ be the bottom automata of the amplifier sequences obtained by applying Theorem~\ref{thm:Amplification} to $G_i^N$.

\begin{proposition}
\label{prop:Independent}
The cellular automaton $\vvarCA_1$ ($\vvarCA_2$) is AND (uniquely ergodic), and it is uniquely ergodic (AND, respectively) if and only if the effective set $N$ is infinite.
In particular, the notions of unique ergodicity and AND are independent.
\end{proposition}

\begin{proof}[Proof sketch]
Consider first the automaton $\vvarCA_1$.
It is easy to see that the proof of Proposition~\ref{prop:DensityConvergence} works with $\varCA_0$ replaced by $\vvarCA_1$, since all the results of Section~\ref{sec:SparseProp} also hold for $G_1$, with Lemma~\ref{lem:Sparsity} replaced by the horizontal version mentioned in the proof of Proposition~\ref{prop:DensityConvergence}, with a slightly weaker bound.
If the effective set $N$ is infinite, the CA is uniquely ergodic, since the proof of Theorem~\ref{thm:MainResult} holds almost as such.
However, if $N \subset [0,n-1]$ for some $n \in \N$, the $n$'th automaton in the amplifier given by Theorem~\ref{thm:Amplification} for $G_1^N$ has a trajectory whose central column contains no blank cells, and then $\vvarCA_1$ is not uniquely ergodic.

The case of $\vvarCA_2$ is similar: All results of Section~\ref{sec:SparseProp} hold for $G_2$, with a weaker bound in Lemma~\ref{lem:Sparsity} caused by the new cells with $\celem$-value different from $\#$, and thus $\vvarCA_2$ is uniquely ergodic by the same proof as $\varCA_0$.
If $N$ is infinite, it is also AND for the same reason as $\varCA_0$, and if $N$ is finite, then some automaton in the amplifier of $G_2^N$ has a trajectory whose central horizontal line contains no blank cells, implying that $\vvarCA_2$ is not AND.
\end{proof}

\input{SparsifiedVarXOR}

\subsection{Asymptotic Sparsity in Other Directions}

As unique ergodicity corresponds to every column of every trajectory being asymptotically sparse, the AND property is its analogue for the horizontal direction.
As a generalization of this idea, we sketch the proof of the analogous result for every rational direction.
Recall that $\sigma : \varalp_0^\Z \to \varalp_0^\Z$ is the shift map, and note that the CA $\sigma^i \circ \varCA_0^k$ has a larger radius than $1$, if $i \neq 0$ or $k > 1$.

\begin{proposition}
\label{prop:AllDirections}
Let $i \in \Z$ and $k > 0$.
Then the cellular automaton $\sigma^i \circ \varCA_0^k$ is uniquely ergodic if and only if $N$ is infinite.
\end{proposition}

\begin{proof}[Proof sketch]
The case of finite $N$ is shown as in Theorem~\ref{thm:MainResult}, so suppose that $N$ is infinite.
First, for all $a \in \Q$ such that $a \notin \{-1,1\}$, we can prove the generalization of Lemma~\ref{lem:Sparsity} where $\eta^s_0$ is replaced by $\eta^{\lfloor as \rfloor}_s$, with the bound depending on $a$, but approaching the original as $a$ approaches $0$.
Then, we again adapt the proof of Theorem~\ref{thm:MainResult}.
This time, we need to show that the asymptotic density of non-blank cells on the line $I = \{\eta^{mk}_{mi} \;|\; m \in \N\}$ is $0$.
For this, we prove the analogue of the first claim, where $K \subset \bigcup_{s \in [0,W_n-1]} [si-4B_{n-1},si+3B_{n-1}] \times [sk-3W_{n-1},sk]$, with the bound $q \in \N$ depending on $i$ and $k$.

Consider then the variant of the second claim where we have redefined $L = \{ t \in [\ell W_{n+1}, (\ell+1)W_{n+1}-1] \;|\; \eta^{tk}_{ti} \neq B_{\varalp_0} \}$.
In its proof, note that the set of coordinates $(s,j) \in \Z^2$ such that $I \cap [j,j+B_n-1] \times [s,s+W_n-1] \neq \emptyset$ approximates the discrete line $J = \{ (r,\lfloor \frac{k W_n}{i B_n} r \rfloor) \;|\; r \in \N \}$.
Then the generalization of Lemma~\ref{lem:Sparsity} we proved earlier, together with the fact that $\frac{k W_n}{i B_n}$ converges to $0$ as $n$ grows, allows us to prove an upper bound for $|L|$, which approaches $\frac{EqW_{n+1}}{n}$ for some constant $E \in \N$ (depending on $i$ and $k$) as $n$ grows.
The proof is finished as in Theorem~\ref{thm:MainResult}.
\end{proof}

\subsection{Computability}
\label{sec:Computability}

In this section, we turn to decision problems concerning unique ergodicity and the AND property.
We start by proving their undecidability, which is the reason for defining the effective set $N$.
Note that a $\Pi^0_2$-complete problem is in particular undecidable.

\begin{proposition}
It is $\Pi^0_2$-complete whether a given cellular automaton, or a given AND cellular automaton, is uniquely ergodic.
Similarly, it is $\Pi^0_2$-complete whether a given cellular automaton, or a given uniquely ergodic cellular automaton, is AND.
\end{proposition}

\begin{proof}
First, let $M$ be a Turing machine, and let $q$ be a state of $M$.
Let $N \subset \N$ be the set of numbers $n$ such that $M$, when run on empty input, enters the state $q$ on its the $n$'th step.
Then $N$ is decidable in polynomial time.

Consider first the problem of determining whether a given CA is uniquely ergodic.
Apply Theorem~\ref{thm:Amplification} to the partial sparsification $G_s^N$, and consider the bottom automaton $\varCA_0$.
By Theorem~\ref{thm:MainResult}, $\varCA_0$ is uniquely ergodic if and only if $N$ is infinite.
Since the infinity of $N$, when $M$ is given, is well known to be $\Pi^0_2$-hard, so is the unique ergodicity of a given CA.
The $\Pi^0_2$-hardness of the other properties follows similarly, using the transformations $G_1^N$ and $G_2^N$, together with Propositions~\ref{prop:DensityConvergence} and~\ref{prop:Independent}.

For the other direction, Proposition~\ref{prop:UEChara} implies that the unique ergodicity of a cellular automaton $\CA : \alp^\Z \to \alp^\Z$ is equivalent to the formula
\[ \forall k \in \N : \exists \ell \in \N : \forall w \in \alp^{2 \ell+1} : |\{ \CA^n(w)_{\ell-n-1} \neq B_\alp \;|\; n \in [0,\ell-1] \}| \leq \frac{\ell}{k}, \]
while asymptotic nilpotency in density is equivalent to the formula
\[ \forall k \in \N : \exists n, \ell \in \N : \forall w \in \CA^n(\alp^{2(n+\ell)+1}) : |\{ w_i \neq B_\alp \;|\; i \in [0,\ell-1] \}| \leq \frac{2\ell+1}{k} \]
by the proof of Lemma~\ref{lem:AND}.
Since these are both $\Pi^0_2$ formulae, the properties are $\Pi^0_2$-complete.
\end{proof}

It is a well known theorem, first proved in \cite{Ka92}, that the nilpotency of a given cellular automaton is undecidable.
We strengthen this result by restricting to the class of uniquely ergodic and AND cellular automata.

\begin{proposition}
It is undecidable whether a given uniquely ergodic and AND cellular automaton is nilpotent.
\end{proposition}

\begin{proof}
Let $M$ be a Turing machine, and let $p_0$ be a program for the trivial cellular automaton on $\{0,1\}^\Z$ that sends everything to $\INF 0 \INF$.
Define the polynomial CA transformation $G^M_s$ by
\[
G^M_s(n, p) = \left\{ \begin{array}{cl}
	p_0, & \mbox{if $M$ halts after $n$ steps on empty input,} \\
	G_s(n, p), & \mbox{otherwise.}
\end{array}\right.
\]
Let $(\vvarCA_n)_{n \in \N}$ be the amplifier sequence of automata given by Theorem~\ref{thm:Amplification} for $G^M_s$.
Now, if $M$ never halts, then $G^M_s = G_s$, and Theorem~\ref{thm:MainResult} and Proposition~\ref{prop:DensityConvergence} imply that $\vvarCA_0$ is uniquely ergodic and AND, but not nilpotent.
However, if $M$ halts on the $n$'th step of its computation, then $\vvarCA_n$ is the trivial automaton, and Corollary~\ref{cor:Parents} applied to $n$ (which is valid up to level $n$) implies that $\vvarCA_0$ is nilpotent.
Thus we have reduced the halting problem to our restricted nilpotency problem.
\end{proof}

\section{Conclusions and Future Directions}
\label{sec:Conclu}

In this article, we have presented a uniquely ergodic cellular automaton which is not nilpotent.
The construction uses a technique we call amplification, which combines self-simulation with arbitrary modification of CA rules.
We have also presented some further results that are given by from minor variants of the construction, as well as related undecidability results.

Our combinatorial characterization of unique ergodicity, Corollary~\ref{cor:UEChara}, only applies to cellular automata with a quiescent state.
Not all cellular automata $\CA : \alp^\Z \to \alp^\Z$ have one, but in general there are states $q_0, \ldots, q_{k-1} \in \alp$ such that $\CA(\INF q_i^\infty) = \INF q_{i+1 \bmod k}^\infty$ for all $i \in [0, k-1]$, and if $\mu_i$ is the Dirac measure concentrated on $\INF q_i^\infty$, then $\frac{1}{k} \sum_{i=0}^{k-1} \mu_i$ is an invariant measure of $\CA$.
It is likely that the characterization can be extended to this general case, and that uniquely ergodic cellular automata without a quiescent state can also be constructed.
However, this seems to require some entirely new ideas, and an even more complicated automaton.

Every uniquely ergodic cellular automaton $\varCA_0 : \varalp_0^\Z \to \varalp_0^\Z$ with a quiescent state $B_{\varalp_0}$ also has the following property.
Let $\CA : \alp^\Z \to \alp^\Z$ be an arbitrary cellular automaton, and consider the product alphabet $\vvaralp = \varalp_0 \times \alp$.
Let $\vvarCA : \vvaralp^\Z \to \vvaralp^\Z$ be any CA that behaves as $\varCA_0$ on the first component of $\vvaralp$, and as $\CA$ on the second component of $\{B_{\varalp_0}\} \times \alp$.
Then we have $\M_\vvarCA = \{\mu_0\} \times \M_\CA$, where $\mu_0$ is the Dirac measure on $\varalp_0$ concentrated on the uniform configuration $\INF B_{\varalp_0}^\infty$.
In other words, $\vvarCA$ has essentially the same set of invariant measures as $\CA$.
However, if we can construct $\vvarCA$ so that it has some desired property, its existence proves that having that property does not essentially restrict the set of invariant measures of a CA.
An obvious candidate would be (some suitable variant of) computational universality, but there are undoubtedly other properties for which the above scheme works too.

Finally, it would be interesting to modify our construction to make the resulting CA reversible.
Of course, a reversible CA cannot be uniquely ergodic, since it preserves the uniform Bernoulli measure, but it might have some weaker property that would shed some light on the possible sets of invariant measures of reversible cellular automata.

\section*{Acknowledgments}

I am thankful to my advisor Jarkko Kari for his advice and support, Ville Salo for many useful discussions on the topics of this article, Charalampos Zinoviadis for directing me to \cite{Ga01} and discussing future directions of this research, Pierre Guillon for his interest in this project, and the anonymous referee for their comments that greatly helped to improve the readability of this article.

\bibliographystyle{plain}
\bibliography{../../../bib/bib}{}

\end{document}

%% file: SparsifiedXOR2.tex
\begin{figure}[htp]
\begin{center}
\begin{tikzpicture}[xscale=2,yscale=2]

\begin{scope}\clip (1,0) rectangle (6,3);
\foreach \x/\y in {0/0,1/0,3/0,3/2,4/0,4/1,4/2,5/0,5/2,6/2}{
\begin{scope}[shift={(\x,\y)}]
\begin{scope}[xscale={1/9},yscale={1/9}]
\fill (4,0) rectangle ++(1,1);
\foreach \y in {1,...,8}{
	\filldraw[fill=gray!50] (4+\y,\y) rectangle ++(1,1);
	\fill (5+\y,\y+1/2) --
	      (4+\y,\y+3/4) --
	      (4+\y,\y+1/4) -- cycle;
	\filldraw[fill=gray!50] (4-\y,\y) rectangle ++(1,1);
	\fill (4-\y,\y+1/2) --
	      (5-\y,\y+3/4) --
	      (5-\y,\y+1/4) -- cycle;
}
\foreach \y in {1,...,4}{
	\filldraw[fill=gray!50] (8-\y,4+\y) rectangle ++(1,1);
	\fill (8-\y+1/2,5+\y) --
	      (8-\y+1/4,4+\y) --
	      (8-\y+3/4,4+\y) -- cycle;
}
\end{scope}

\end{scope}
}
\foreach \x/\y in {0/1,0/2,1/1,1/2,2/0,2/1,2/2,3/1,5/1,6/0,6/1}{
\begin{scope}[shift={(\x,\y)}]
\begin{scope}[xscale={1/9},yscale={1/9}]
\draw (4,0) rectangle ++(1,1);
\end{scope}

\end{scope}
}
\end{scope}\draw (1,0) rectangle (6,3);
\end{tikzpicture}
\end{center}
\caption{The sparsification transformation applied to the three-neighbor XOR automaton, with $n = 4$. The cells with left, right and up arrows are left moving, right moving and returning, respectively. The black cells are right moving with $\cou = 0$, and white cells are blank. The live cells on each row have the same counter value.}
\label{fig:SparsifiedXOR}
\end{figure}

%% file: SimulationPic.tex
\begin{figure}[ht]
\begin{center}
\begin{tikzpicture}

\begin{scope}[shift={(4.5,-3+.625)},xscale=0.3,yscale=1.2]

\foreach \x/\y in {0/2,1/0,1/1,1/2,1/4,2/3,2/4,3/0,3/1,3/4,4/1,5/4,6/1,6/3,6/4,7/0,7/1,7/2,7/4,8/1,8/2,9/1,10/0,10/2}{
\begin{scope}[shift={(\x,\y)}]
\fill[color=gray] (0,0) rectangle (1,1);
\draw[densely dotted] (0,.25) -- (1,.25);
\draw[color=gray!50!black] (0,.25) -- ++(.5,.125) -- ++(-.25,.125) -- ++(.5,.125) -- ++(-.75,.125) -- ++(.5,.125) -- ++(-.5,.125);
\draw (0,0) rectangle (1,1);

\end{scope}
}
\foreach \x/\y in {0/3,3/2,4/2,4/3,5/0,5/2,8/3,9/3,9/4,10/3,10/4}{
\begin{scope}[shift={(\x,\y)}]

\end{scope}
}
\foreach \x/\y in {0/0,0/1,0/4,1/3,5/3,6/0,6/2,9/0}{
\begin{scope}[shift={(\x,\y)}]
\fill[color=gray!50] (0,.5) -- (0,1) -- (1,1) -- (1,0) -- (0,.25);
\draw[densely dotted] (0,.25) -- (1,.25);
\draw[color=gray] (0,.25) -- ++(.5,.125) -- ++(-.25,.125) -- ++(.5,.125) -- ++(-.75,.125) -- ++(.5,.125) -- ++(-.5,.125);
\draw (0,.25) -- (0,1) -- (1,1) -- (1,0) -- (0,.25);

\end{scope}
}
\foreach \x/\y in {2/2,3/3,4/0,7/3,8/0,8/4,10/1}{
\begin{scope}[shift={(\x,\y)}]
\fill[color=gray!50] (0,0) -- (0,1) -- (1,1) -- (1,.25) -- (0,0);
\draw[densely dotted] (0,.25) -- (1,.25);
\draw[color=gray] (0,.25) -- ++(.5,.125) -- ++(-.25,.125) -- ++(.5,.125) -- ++(-.75,.125) -- ++(.5,.125) -- ++(-.5,.125);
\draw (0,0) -- (0,1) -- (1,1) -- (1,.25) -- (0,0);

\end{scope}
}
\foreach \x/\y in {2/0,2/1,4/4,5/1,9/2}{
\begin{scope}[shift={(\x,\y)}]
\fill[color=gray!50] (0,0) -- (0,1) -- (1,1) -- (1,0) -- (.5,.125) -- (0,0);
\draw[densely dotted] (0,.25) -- (1,.25);
\draw[color=gray] (0,.25) -- ++(.5,.125) -- ++(-.25,.125) -- ++(.5,.125) -- ++(-.75,.125) -- ++(.5,.125) -- ++(-.5,.125);
\draw (0,0) -- (0,1) -- (1,1) -- (1,0) -- (.5,.125) -- (0,0);

\end{scope}
}
\draw (0,0) rectangle (11,5);

\end{scope}

\foreach \x/\y in {
1/0,3/0,7/0,10/0,
1/1,3/1,4/1,6/1,7/1,8/1,9/1,
0/2,1/2,7/2,8/2,10/2,
2/3,6/3,
1/4,2/4,3/4,5/4,6/4,7/4
}{
\fill[gray] (.25*\x,.25*\y) rectangle ++(.25,.25);
}
\draw[step=.25] (0,0) grid (2.75,1.25);

\node () at (3.6,.625) {$\stackrel{\pi}{\longleftarrow}$};

\end{tikzpicture}
\end{center}
\caption{Simulating the three-neighbor XOR automaton. The dotted lines represent transitions between retrieval and computation phases, and the jagged lines represent the agents. Dark macro-cells have state $1$, and light ones have $0$.}
\label{fig:Simulation}
\end{figure}

%% file: OverlapLemma2.tex
\begin{figure}
\begin{center}
\begin{tikzpicture}

\draw[fill=gray,thick] (.3,.4) rectangle ++(1.2,4);
\draw[fill=gray!50,thick] (0,0) rectangle ++(1.2,4);
\begin{scope}
\clip (0,0) rectangle ++(1.2,4);
\draw[thick,dashed] (.3,.4) rectangle ++(1.2,4);
\end{scope}

\node () [left] at (0,2) {$P$};
\node () [right] at (1.5,2.2) {$R$};

\draw[thick,densely dotted] (.9,3) -- (.9,1.2);
\draw[thick,densely dotted] (.4,3.4) -- (.4,2.6) -- (1.4,1.6);

\draw (.75,1.9) rectangle ++(.3,.5);
\draw (.75,1.9) -- (3,0);
\draw (.75,2.4) -- (3,5);

\begin{scope}[shift={(3,0)}]
\clip (0,0) rectangle (3,5);

\foreach \l in {-1,...,7}{
	\draw[fill=gray,thick] (.1+\l*.3,3.7-\l*.6) rectangle ++(.3,.6);
	\draw[fill=white,thick] (.4+\l*.3,3.1-\l*.6) -- ++(0,.6) -- ++(-.3,0) -- ++(0,-.4) -- cycle;
}

\foreach \l in {-1,...,5}{
	\draw[fill=gray!50,thick] (2.6-\l*.3,2.6-\l*.6) rectangle ++(.3,.6);
	\draw[fill=white,thick] (2.6-\l*.3,2.0-\l*.6) -- ++(0,.6) -- ++(.3,0) -- ++(0,-.4) -- cycle;
}
\foreach \l in {1,...,4}{
	\draw[fill=gray!50,thick] (2.6-\l*.3,2.6+\l*.6) rectangle ++(.3,.6);
	\draw[fill=white,thick] (2.9-\l*.3,2.0+\l*.6) -- ++(0,.6) -- ++(-.3,0) -- ++(0,-.4) -- cycle;
}

\draw[thick,dashed] (1.6,.7) rectangle ++(.3,.6);

\end{scope}

\draw (3,0) rectangle ++(3,5);

\end{tikzpicture}
\end{center}
\caption{A visualization of Lemma~\ref{lem:NoOverlap}. On the left, the overlapping macro-cells $P$ and $R$, and the two chains of subcells. On the right, the non-blank macro-cells $P^{(s)}_{(j_s)}$ and $R^{(s)}_{(k_s)}$, and under them, the macro-cells $P^{(s-1)}_{(j_s)}$ and $R^{(s-1)}_{(k_s)}$, of the simulation $S_{n-1}$. Of these, some necessarily overlap, shown by the dashed border.}
\label{fig:OverlapLemma}
\end{figure}
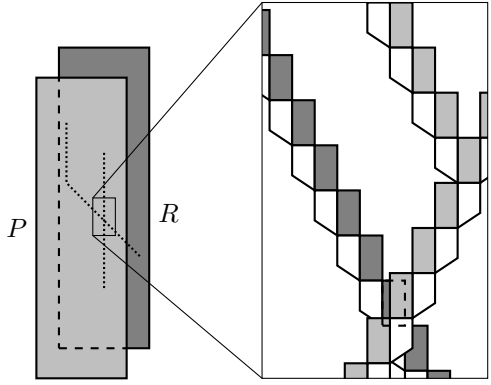

%% file: LnPic.tex
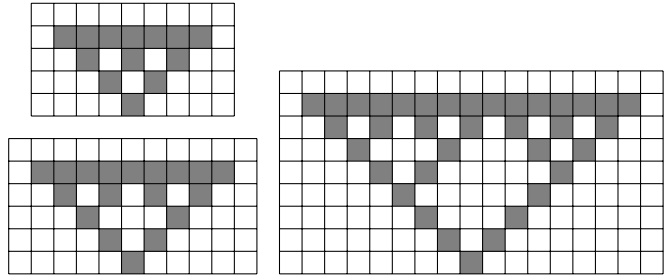
\begin{figure}[htp]
\begin{center}
\begin{tikzpicture}[xscale=0.3,yscale=0.3]

\foreach \x/\y in {1/4,2/3,2/4,2/10,3/2,3/4,3/9,3/10,4/1,4/3,4/4,4/8,4/10,5/0,5/4,5/7,5/9,5/10,6/1,6/3,6/4,6/8,6/10,7/2,7/4,7/9,7/10,8/3,8/4,8/10,9/4,13/7,14/6,14/7,15/5,15/7,16/4,16/6,16/7,17/3,17/7,18/2,18/4,18/6,18/7,19/1,19/5,19/7,20/0,20/6,20/7,21/1,21/7,22/2,22/6,22/7,23/3,23/5,23/7,24/4,24/6,24/7,25/5,25/7,26/6,26/7,27/7}{
\begin{scope}[shift={(\x,\y)}]
\fill[gray] (0,0) rectangle (1,1);

\end{scope}
}
\foreach \x/\y in {0/0,0/1,0/2,0/3,0/4,0/5,1/0,1/1,1/2,1/3,1/5,1/7,1/8,1/9,1/10,1/11,2/0,2/1,2/2,2/5,2/7,2/8,2/9,2/11,3/0,3/1,3/3,3/5,3/7,3/8,3/11,4/0,4/2,4/5,4/7,4/9,4/11,5/1,5/2,5/3,5/5,5/8,5/11,6/0,6/2,6/5,6/7,6/9,6/11,7/0,7/1,7/3,7/5,7/7,7/8,7/11,8/0,8/1,8/2,8/5,8/7,8/8,8/9,8/11,9/0,9/1,9/2,9/3,9/5,9/7,9/8,9/9,9/10,9/11,10/0,10/1,10/2,10/3,10/4,10/5,12/0,12/1,12/2,12/3,12/4,12/5,12/6,12/7,12/8,13/0,13/1,13/2,13/3,13/4,13/5,13/6,13/8,14/0,14/1,14/2,14/3,14/4,14/5,14/8,15/0,15/1,15/2,15/3,15/4,15/6,15/8,16/0,16/1,16/2,16/3,16/5,16/8,17/0,17/1,17/2,17/4,17/5,17/6,17/8,18/0,18/1,18/3,18/5,18/8,19/0,19/2,19/3,19/4,19/6,19/8,20/1,20/2,20/3,20/4,20/5,20/8,21/0,21/2,21/3,21/4,21/5,21/6,21/8,22/0,22/1,22/3,22/4,22/5,22/8,23/0,23/1,23/2,23/4,23/6,23/8,24/0,24/1,24/2,24/3,24/5,24/8,25/0,25/1,25/2,25/3,25/4,25/6,25/8,26/0,26/1,26/2,26/3,26/4,26/5,26/8,27/0,27/1,27/2,27/3,27/4,27/5,27/6,27/8,28/0,28/1,28/2,28/3,28/4,28/5,28/6,28/7,28/8}{
\begin{scope}[shift={(\x,\y)}]

\end{scope}
}
\foreach \x/\y in {0/6,0/7,0/8,0/9,0/10,0/11,1/6,2/6,3/6,4/6,5/6,6/6,7/6,8/6,9/6,10/6,10/7,10/8,10/9,10/10,10/11,11/0,11/1,11/2,11/3,11/4,11/5,11/6,11/7,11/8,11/9,11/10,11/11,12/9,12/10,12/11,13/9,13/10,13/11,14/9,14/10,14/11,15/9,15/10,15/11,16/9,16/10,16/11,17/9,17/10,17/11,18/9,18/10,18/11,19/9,19/10,19/11,20/9,20/10,20/11,21/9,21/10,21/11,22/9,22/10,22/11,23/9,23/10,23/11,24/9,24/10,24/11,25/9,25/10,25/11,26/9,26/10,26/11,27/9,27/10,27/11,28/9,28/10,28/11}{
\begin{scope}[shift={(\x,\y)}]

\end{scope}
}
\draw (0,0) grid (11,6);\draw (1,7) grid (10,12);\draw (12,0) grid (29,9);
\end{tikzpicture}
\end{center}
\caption{The sets $L(4)$ (top left), $L(5)$ (bottom left) and $L(7)$ (right). Note how the two smaller sets are used to construct $L(7)$.}
\label{fig:LnPic}
\end{figure}

%% file: SparsifiedVarXOR.tex
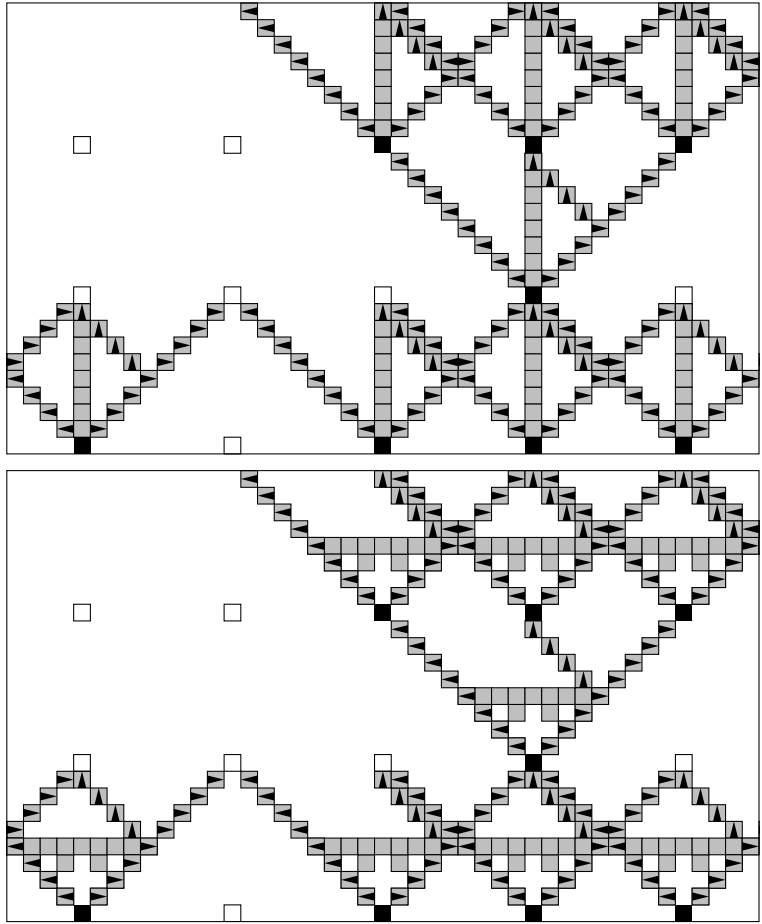
\begin{figure}[htp]
\begin{center}
\begin{tikzpicture}[xscale=2.0,yscale=2.0]

\begin{scope}\clip (1,0) rectangle (6,6+1/9);
\foreach \x/\y in {0/3,1/3,3/3,3/5,4/3,4/4,4/5,5/3,5/5,6/5}{
\begin{scope}[shift={(\x,\y)}]
\begin{scope}[xscale={1/9},yscale={1/9},shift={(0,1)}]
\fill (4,0) rectangle ++(1,1);
\foreach \y in {1,...,8}{
	\filldraw[fill=gray!50] (4+\y,\y) rectangle ++(1,1);
	\fill (5+\y,\y+1/2) --
	      (4+\y,\y+3/4) --
	      (4+\y,\y+1/4) -- cycle;
	\filldraw[fill=gray!50] (4-\y,\y) rectangle ++(1,1);
	\fill (4-\y,\y+1/2) --
	      (5-\y,\y+3/4) --
	      (5-\y,\y+1/4) -- cycle;
	\filldraw[fill=gray!50] (4,\y) rectangle ++(1,1);
}
\foreach \y in {1,...,4}{
	\filldraw[fill=gray!50] (8-\y,4+\y) rectangle ++(1,1);
	\fill (8-\y+1/2,5+\y) --
	      (8-\y+1/4,4+\y) --
	      (8-\y+3/4,4+\y) -- cycle;
}
\end{scope}

\end{scope}
}
\foreach \x/\y in {0/0,1/0,3/0,3/2,4/0,4/1,4/2,5/0,5/2,6/2}{
\begin{scope}[shift={(\x,\y)}]
\begin{scope}[xscale={1/9},yscale={1/9}]
\foreach \y in {1,...,4}{
	\filldraw[fill=gray!50] (8-\y,4+\y) rectangle ++(1,1);
	\fill (8-\y+1/2,5+\y) --
	      (8-\y+1/4,4+\y) --
	      (8-\y+3/4,4+\y) -- cycle;
}
\foreach \x/\y in {-1/1,1/1,-2/2,2/2,-3/3,-1/3,1/3,3/3,-4/4,-3/4,-2/4,-1/4,0/4,1/4,2/4,3/4,4/4}{
	\filldraw[fill=gray!50] (\x+4,\y) rectangle ++(1,1);
}
\fill (4,0) rectangle ++(1,1);
\foreach \y in {1,...,8}{
	\filldraw[fill=gray!50] (4+\y,\y) rectangle ++(1,1);
	\fill (5+\y,\y+1/2) --
	      (4+\y,\y+3/4) --
	      (4+\y,\y+1/4) -- cycle;
	\filldraw[fill=gray!50] (4-\y,\y) rectangle ++(1,1);
	\fill (4-\y,\y+1/2) --
	      (5-\y,\y+3/4) --
	      (5-\y,\y+1/4) -- cycle;
}
\end{scope}

\end{scope}
}
\foreach \x/\y in {0/4,0/5,1/4,1/5,2/3,2/4,2/5,3/4,5/4,6/3,6/4}{
\begin{scope}[shift={(\x,\y)}]
\begin{scope}[xscale={1/9},yscale={1/9},shift={(0,1)}]
\draw (4,0) rectangle ++(1,1);
\end{scope}

\end{scope}
}
\foreach \x/\y in {0/1,0/2,1/1,1/2,2/0,2/1,2/2,3/1,5/1,6/0,6/1}{
\begin{scope}[shift={(\x,\y)}]
\begin{scope}[xscale={1/9},yscale={1/9}]
\draw (4,0) rectangle ++(1,1);
\end{scope}

\end{scope}
}
\end{scope}\draw (1,0) rectangle (6,3);\draw (1,3+1/9) rectangle (6,6+1/9);
\end{tikzpicture}
\end{center}
\caption{The two modifications of $G_s$, applied to the three-neighbor XOR automaton with $n = 4$, with $G_1$ on top and $G_2$ below. The unmarked cells are not live, but have nonzero $\ccou$-values and $\celem$-values, respectively. Compare with Figure~\ref{fig:SparsifiedXOR}.}
\label{fig:SparsifiedVarXOR}
\end{figure}